\documentclass{article}
\usepackage[utf8]{inputenc}
\usepackage[a4paper,left = 1.2in,right = 1.2in, top = 1.5in, bottom = 1.2in]{geometry}
\usepackage{amsmath}
\usepackage{amsthm}
\usepackage{bm}
\usepackage{dsfont}
\usepackage{palatino}
\usepackage{cancel}
\usepackage{amssymb}
\usepackage{commath}
\usepackage{mathtools}
\usepackage{float}
\usepackage{algorithm}
\usepackage{cases}

\usepackage[numbers]{natbib}

\def\be{\begin{equation}}
\def\ee{\end{equation}}

\def\bea{\begin{align}}
\def\eea{\end{align}}

\def\bea*{\begin{align*}}
\def\eea*{\end{align*}}
\theoremstyle{plain}

\newtheorem{proposition}{Proposition}

\theoremstyle{definition} 
\newtheorem{definition}{Definition}
\newtheorem{example}{Example}
\newtheorem{remark}{Remark}

\DeclareMathOperator{\E}{\mathbb{E}}

\DeclareMathOperator{\N}{\mathbb{N}}

\DeclareMathOperator{\Pb}{\mathbb{P}}

\DeclareMathOperator{\R}{\mathbb{R}}

\DeclareMathOperator{\calA}{\mathcal{A}}

\DeclareMathOperator{\calC}{\mathcal{C}}

\DeclareMathOperator{\calH}{\mathcal{H}}

\DeclareMathOperator{\calK}{\mathcal{K}}
\DeclareMathOperator{\calL}{\mathcal{L}}

\DeclareMathOperator{\calP}{\mathcal{P}}

\DeclareMathOperator{\calR}{\mathcal{R}}

\usepackage{subcaption}
\usepackage{graphicx}
\usepackage{color}

\usepackage{xcolor}
\usepackage[colorlinks,linkcolor=blue,citecolor=blue,urlcolor = violet]{hyperref}
\usepackage[capitalise]{cleveref}

\usepackage{fancyhdr}
\title{Deep Level-set Method for  Stefan Problems}

\author{Mykhaylo \ Shkolnikov\footnote{ 
{\tt mykhaylo@princeton.edu}}
\and H. Mete Soner\footnote{
{\tt soner@princeton.edu}}
\and Valentin Tissot-Daguette\footnote{ 
{\tt v.tissot-daguette@princeton.edu}}
\\[0.8em]
Department of Operations Research and Financial
Engineering\\ Princeton University
}

\begin{document}
\maketitle

\vspace{0mm}
\begin{abstract}
We propose a level-set approach to characterize the region occupied by the solid in Stefan problems with and without surface tension, based on their recent probabilistic reformulation. The level-set function is parameterized by a feed-forward neural network, whose parameters are trained using the probabilistic formulation of the Stefan growth condition. The algorithm can handle Stefan problems where the liquid is supercooled and can capture surface tension effects through the simulation of particles along the moving boundary together with an efficient approximation of the mean curvature. We demonstrate the effectiveness of the method on a variety of examples with and without radial symmetry. 
\end{abstract}
\vspace{2mm}
	
\\ \vspace{-0mm}



\section{Introduction}\label{sec:intro}

The Stefan problem \cite{LC,Stefan1,Stefan2,Stefan3,Stefan4} is central to partial differential equations involving free boundaries. It aims to capture the moving interface separating a solid from a liquid region, as well as the evolution of the temperature in both regions. Despite its simple description and many deep results obtained since its introduction (see, e.g.,  \cite{Figalli2018} and the references therein), many intriguing  questions remain open. In particular, weak solutions to the Stefan problem are non-unique in general, while strong solutions may fail to exist. Thus, further restrictions are needed to obtain a unique characterization. A recent approach developed by Delarue, Guo, Nadtochiy and the first author \cite{DNS,MishaScaling,guo2023stefan} provides stochastic representations and proposes the notion of \textit{physical probabilistic} solutions as a selection principle.
This new notion is expected to lie between weak and strong solutions, as shown in \cite{MishaScaling} for the one-phase supercooled Stefan problem.

Probabilistic solutions satisfy a growth condition relating the change in volume of the solid region, denoted by $\Gamma_t$, to the proportion of absorbed ``heat'' particles in the two phases:
\begin{equation}\label{eq:probaCond2}
\underbrace{|\Gamma_{0}| -  |\Gamma_t|}_{\text{Volume change of the solid}}
   = \; \underbrace{\eta\Pb(\tau^1 \le t) - \Pb(  \tau^2 \le t)}_{\text{Absorbed liquid $\&$ solid heat particles}},   
\end{equation}\\[-1em]
where $\tau^1$ (resp.~$\tau^2$) stands for the hitting time of the moving interface $\partial \Gamma = (\partial \Gamma_t)_{t \in [0,T]}$ for particles in the liquid (resp.~solid). 
The binary parameter $\eta$ captures the effect of particles in the liquid when the latter  has a nonnegative temperature ($\eta = 1$) or is \textit{supercooled} ($\eta = -1$). An exact statement is given in \cref{def:probSolNoTens,def:probSolTension}, below.  While probabilistic solutions of the Stefan problem are quite well understood in one space dimension \cite{DNS,CRS} and for  radially symmetric solids \cite{guo2023stefan,MishaRadialTension}, less is known in the general case. Additional challenges arise  when  incorporating surface tension effects through the classical \textit{Gibbs-Thomson law}, \cite{Almgren,MishaRadialTension,Popinet}, which postulates that the temperature at the interface is below (resp. above) the equilibrium melting point where the solid is locally convex (resp.~concave).   

We leverage the \text{probabilistic} solutions of Stefan problems, the celebrated level-set method of Osher and Sethian \cite{OsherSethian}, and the recent advances in training neural networks in order to produce efficient~numerical algorithms. More specifically, we represent $\Gamma = (\Gamma_t)_{t\in [0,T]}$, the evolving region occupied by the solid, by a level-set function which is parameterized by a feed-forward neural network. That is, the evolving solid is given by the zero sublevel set of an appropriate function $\Phi: [0,T] \times \R^d \to \R$, focusing on $d\in \{2,3\}$ in our  numerical experiments. 
The level-set method is widely used to describe the evolution of moving  interfaces in arbitrary dimensions as it imposes no assumptions on the geometry of the unknown region and can easily handle changes in topological properties. For example, level sets are able to capture the separation of a connected set into several components and vice versa. As demonstrated in the numerical examples below (\cref{sec:Dumbbell}), this flexibility is of crucial importance for general Stefan problems.  We refer the reader to the excellent books by \citet{OsherFedkiw} and \citet{Sethian} for a comprehensive description of level-set methods.  

To the best of our knowledge, the use of the level-set method for the Stefan problem was first proposed in \cite{OsherStefan}.  Their algorithm alternately approximates the moving interface through level-set functions and the temperature in the two phases via a finite difference scheme for the heat equation. The method is capable of accurately reproducing known solutions to the Stefan problem, as well as of generating realistic dendritic growth for a variety of solids (see also \cite{Gibou}, focusing on dendritic crystallization, and \cite[Section 23]{OsherFedkiw}, presenting applications to general heat flows). Herein, the level-set function is parameterized by a time-space feedforward neural network $\Phi:[0,T]\times \R^d \times \Theta \to \R$ where $\Theta$ is a finite-dimensional  parameter set.  The parameters are trained using stochastic gradient descent by converting the growth condition of probabilistic solutions 
into a loss function. The computation of the latter involves the simulation of reflected Brownian particles in the two phases. The proposed method has the advantage that the normal vector to the interface -- which is essential in level-set methods -- can be effortlessly computed through automatic differentiation of the deep level-set function. 

The training of neural network parameters is achieved by minimizing a loss function that involves stopped particles. Since the hitting times of sharp interfaces lead to vanishing gradient issues and prevent the training of the deep level-set function, we use a relaxation procedure as in \cite{ReppenSonerTissotDF,ReppenSonerTissotFB,SonerTissot}, developed for optimal stopping. This relaxation consists of  introducing a \textit{mushy region} separating the solid and the liquid as in phase-field models \cite{BSS,B-mushy,SonerMS}, and then using stopping probabilities. The loss function defined in \eqref{eq:loss}, below, tries to enforce the local version \eqref{eq:probaCond} of \eqref{eq:probaCond2} for a large number of randomly chosen test functions. Such a construction based on an identity holding for a class of test functions is novel and may find other applications.

The trained network encapsulates the two phases, and it is important to emphasize that the temperature function does not need to be approximated during training. Due to the probabilistic nature of the algorithm, the temperature function can be estimated later through the empirical measure of the surviving particles. Moreover, surface tension effects can be seamlessly integrated into the algorithm, see \cref{sec:MethodTension}. We refer the interested reader to \cite{Popinet} surveying computational methods for differential equations with surface tension in fluid mechanics.

Combining the level-set method with deep learning to solve free boundary problems appears to be new, although our work shares some similarities with \cite{Angelo}, which provides a numerical resolution of controlled front propagation flows with the level-set method, and where the velocity is approximated by neural networks and not the level-set function. The present approach is motivated by the recent advances of deep learning to solve complex and/or high-dimensional problems in partial differential equations \cite{HanJentzenE,DGM,WangPerdikaris,PhamPDE}, optimal stopping \cite{Becker1,Becker2,ReppenSonerTissotDF,ReppenSonerTissotFB}, as well as general stochastic control problems \cite{PhamNum,HanE,ReppenSonerTissotDF}. In the case of physical phenomena, \textit{physics-informed} neural networks \cite{PerdikarisPI,WangPerdikaris} successfully combine observed data with known physical laws to learn the solution for a variety of problems in physics. 
The neural parameterization of level-set functions has already shown promising results in computer vision. Among others, a neural network is trained to approximate the signed distance function associated with three-dimensional objects in \cite{DeepSDF} and their \emph{occupancy probability} in \cite{OccupancyNN}.  

We believe that the proposed deep level-set method can be applied to problems in a variety of contexts. This includes  free boundary problems in physics, such as  Hele-Shaw and Stokes flows \cite{Howison}.  In mathematical finance, the method can be used to generalize the \emph{neural optimal stopping boundary method} of \cite{ReppenSonerTissotFB} for the exercise boundary of American options with known geometric structure. Also, optimal portfolio problems with transaction costs \cite{MKRS} could benefit from our method once the so-called no-trade zone is represented by a deep level-set function.\\[-1em]  

\textbf{Structure of the paper.} We introduce the Stefan problem and its probabilistic solution in \cref{sec:Problem}. The deep level-set method is described in \cref{sec:Method} and extended in \cref{sec:MethodTension} for the Stefan problem with  surface tension.     \cref{sec:Numerics} is devoted to the numerical results in the radially symmetric case (\cref{sec:RadialNumerics}) and for  general shapes of the solid (\cref{sec:GeneralNumerics}).  \cref{sec:conclusion} concludes, and \cref{app:Proofs} contains the proofs of the main results. \\[-1em]

\textbf{Notations.}  Given a Lebesgue measurable set $A\subset \R^d$, $d\ge 1$, and a measurable function $\psi:\R^d\to \R$, we employ the shorthand notations $\int_A \psi= \int_A \psi(x)\,\mathrm{d}x$ and $\int_{\partial A} \psi =  \int_{\partial A}\psi(x)\,\mathrm{d}\calH^{d-1}(x)$, where $\calH^{d-1}$ is the $(d-1)$-dimensional Hausdorff measure. We also write $|A|$ for the Lebesgue measure of $A$. 

\newpage
\section{Stefan Problems and Probabilistic Solutions}
\label{sec:Problem}
Let $\Omega$ be a bounded domain in $\R^d$, $d\ge 1$, and $T\in (0,\infty)$. Given a closed subset $\Gamma_{0-} \subseteq \overline{\Omega}$ and $u_0^1:  \overline{\Omega\setminus \Gamma_{0-}} \to \R, \  u_{0}^2: \Gamma_{0-} \to \R_{-}$, the  strong formulation of the  \textit{two-phase Stefan problem} amounts to  finding  a triplet $(u^1,u^2,\Gamma)= (u^1(t,\cdot),u^2(t,\cdot),\Gamma_t)_{t\in [0,T] }$  such that 
\begin{subnumcases}{}
\partial_t u^i = \frac{\alpha_i}{2}\Delta u^i, & on \  int $\Gamma^i :=  \{ (t,x)\in (0,T] \times \Omega: \, x \in \text{int} \ \Gamma^i_t \}, \;  i =1,2,$ \label{eq:2SPHeat}
\\ [0.5em] 
u^i(0-,\cdot)  = u^i_0,  & on  \ $\Gamma^i_{0-}, \; i = 1,2,$\label{eq:2SPInitial}
\\[0.5em]       
\partial_{\mathfrak{n}} u^i = 0,  & on  $\ \partial \Omega \cap \Gamma^i := \{(t,x)\in [0,T] \times \partial\Omega \ : \ x \in  \Gamma^i_t \} , \;  i =1,2, $  \label{eq:2SPNeumann} 
\\[0.5em]
V = \frac{\alpha_2}{2L} \partial_\nu u^2 -  \frac{\alpha_1}{2L}  \partial_\nu u^1,  & on \ $ \partial \Gamma  := \{(t,x)\in [0,T] \times \Omega \ : \ x \in \partial \Gamma_t \} \label{eq:2SPStefan}, $\\[0.5em] 
u^1  = u^2 = 0,  & on  $\partial \Gamma, $ \label{eq:2SPDirichlet} 
\end{subnumcases}
where $\Gamma^1_t = \Omega \setminus \Gamma_t$, $\Gamma^2_t = \Gamma_t$ represent respectively the liquid and the solid region, and $\mathfrak{n}$ is the outward normal vector field on $\partial \Omega$. In \eqref{eq:2SPStefan}, $V$ is the outward normal velocity of $\partial \Gamma_t$, $\nu$ the outward normal vector field on $\partial \Gamma_t$, $L$ the latent heat of fusion, and $\alpha_1,\alpha_2$ the thermal diffusivities. An illustration of the two-phase Stefan problem is given in \cref{fig:PDE}. We suppose that the boundary $\partial \Gamma_t $ is in the solid region so that $\Gamma_t^1$ (resp. $\Gamma_t = \Gamma_t^2$) is relatively open (resp. closed) in $\Omega$.  

The initial temperature in the solid $u_0^2$ is always assumed to be non-positive. At the same time, the temperature in the liquid $u_0^1$ is assumed to be either non-negative or non-positive everywhere. For convenience, we introduce the  parameter $\eta := \text{sign}(u_0^1)$ which indicates whether the liquid is initially regular ($\eta = 1$) or \textit{supercooled} ($\eta = -1$). 

\begin{figure}[h]
\caption{Illustration of the two-phase Stefan problem \eqref{eq:2SPHeat}-\eqref{eq:2SPDirichlet}. Evolution of the liquid region (blue) and solid region (white) in a freezing regime for some time points $0\le s < t \le T$.}
\vspace{-2mm}
\begin{subfigure}[b]{0.49\textwidth}
    \centering   \includegraphics[height=2in,width=2in]{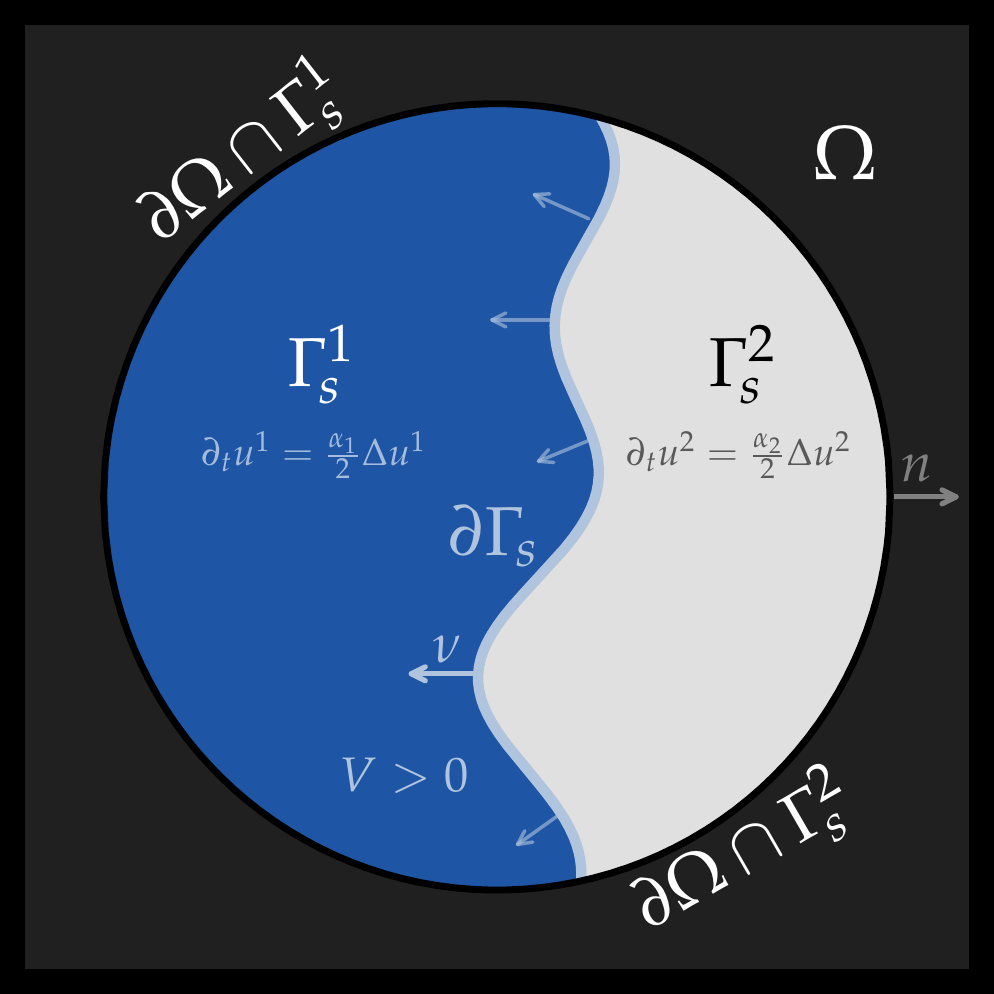}
    \label{fig:PDE1}
\end{subfigure}
\begin{subfigure}[b]{0.49\textwidth}
    \centering \includegraphics[height=2in,width=2in]{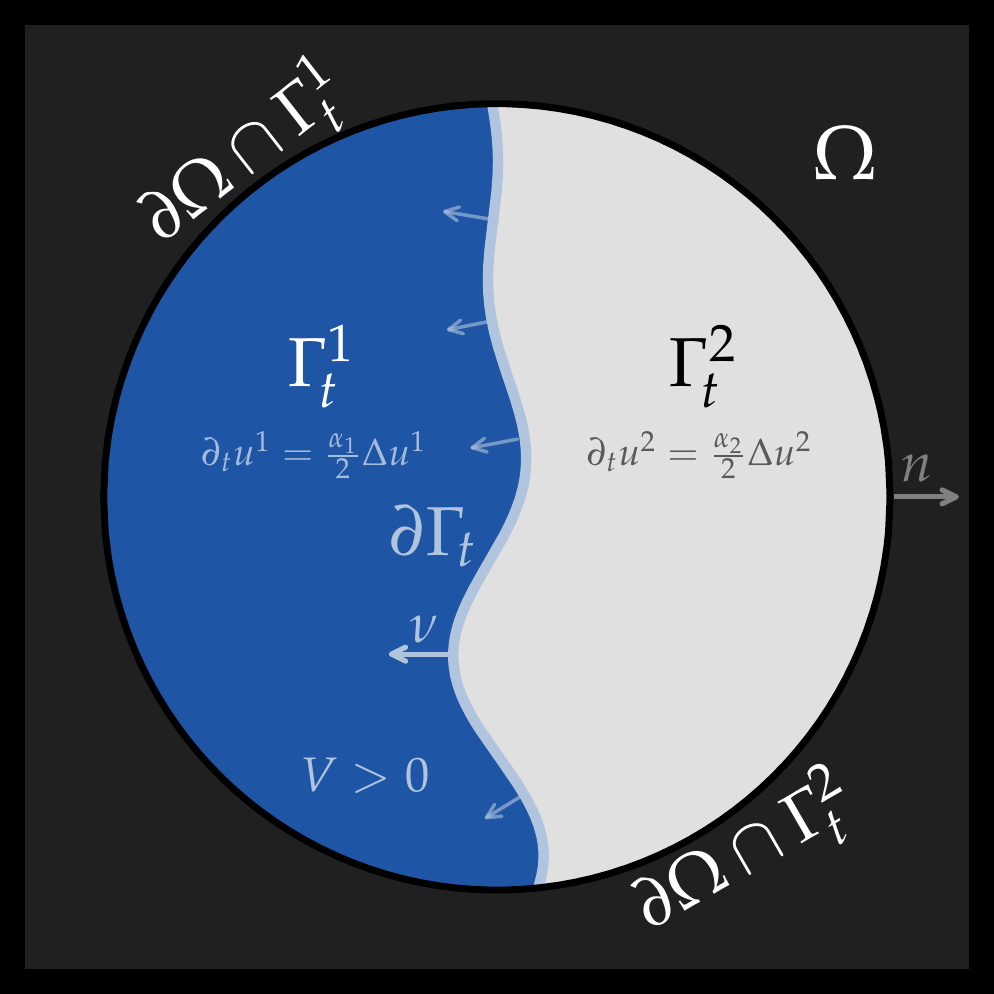}
    \label{fig:PDE2}
\end{subfigure}
\label{fig:PDE}
\end{figure}

\subsection{Probabilistic Solutions}

We assume throughout that $\lVert u_0^i\rVert_{L^1(\Gamma_{0-}^i)} = 1,$ $i=1,2$. This requirement can be easily removed, see Remark \ref{rem:initConstants} below. Further, we let $X^i$ be a  Brownian motion with diffusivity $\alpha_i$ that is normally reflected along $\partial \Omega$ and absorbed when hitting the moving interface $\partial \Gamma$ (in light of \eqref{eq:2SPHeat}, \eqref{eq:2SPNeumann}, \eqref{eq:2SPDirichlet}). More specifically,
\begin{align}\label{eq:rBM}
    X_t^i = X_0^i + \sqrt{\alpha_i} \ W_{t\wedge \tau^i}^i + l_{t\wedge \tau^i}^i,\;0\le t \le T,
    \quad 
    \tau^i = \inf\{t\in [0,T]:\, X_0^i + \sqrt{\alpha_i} \ W_t^i + l_t^i \notin \Gamma_t^i \},
\end{align}
where $W^i$ is a standard Brownian motion in $\R^d$ and $l^i$ the \emph{local time process} at $\partial \Omega$. In addition, we take $X_0^i\sim |u_0^i|(x)\,\mathrm{d}x$. We can now state the definition of a probabilistic solution. 
\begin{definition} \label{def:probSolNoTens}
We say that $(\mu^1,\mu^2,\Gamma)$ is a probabilistic solution of the Stefan problem \eqref{eq:2SPHeat} $\!-$\eqref{eq:2SPDirichlet} if for all $t \in [0,T]$, one has 
\begin{equation}
\label{eq:probaCond}
\int_{\Gamma_{0-}}\psi -  \int_{\Gamma_t  } \psi 
   = \frac{1}{L}\big(\eta \E^{\mu^1}[  \psi(X_{\tau^1}^1)\,\mathds{1}_{\{ \tau^1 \le t\}}] - \E^{\mu^2}[  \psi(X_{\tau^2}^2)\,\mathds{1}_{\{ \tau^2 \le t\}}]\big) , \quad \psi \in \mathcal{C}^{\infty}_c(\Omega),
\end{equation}
where $\mu^i$ is the law of $X^i$, $i=1,2$. 
\end{definition} 

We next prove that the definition of a probabilistic solution is consistent with that of a classical one, as already shown for the one-phase problem in \cite[Proposition 5.5]{MishaScaling}.  

\begin{proposition}\label{prop:GrowthCondNoTension}
Suppose that  $(u^1,u^2,\Gamma)$ is a classical solution of the Stefan problem  \eqref{eq:2SPHeat} $\!-$\eqref{eq:2SPDirichlet}. Let $X^i$ be as in \eqref{eq:rBM} and set $\mu^i=\Pb \circ (X^i)^{-1}$. Then, $( \mu^1,\mu^2 ,\Gamma)$ is a probabilistic solution of \eqref{eq:2SPHeat} $\!-$\eqref{eq:2SPDirichlet}. 
\end{proposition}

\begin{proof}
See \cref{app:GrowthCondNoTension}.
\end{proof}

In the course of the proof of \cref{prop:GrowthCondNoTension}, we find that
\begin{equation}\label{eq:UvsProb}
|u^i|(t,x)\,\mathrm{d}x = \Pb(X_t^i \in \mathrm{d}x,\,\tau^i>t), \quad i=1,2. 
\end{equation}
The temperature, albeit not our primary focus, can thus be retrieved from the (sub)density of ``survived'' particles (i.e., those not yet absorbed by the moving boundary $\partial \Gamma$) once $\Gamma = (\Gamma_t)_{t\in [0,T]}$ has been estimated. Further details are given in \cref{sec:longterm}. 

\begin{remark}\label{rem:identity}
Suppose that the boundary of the ice region lies entirely in $\Omega$ for all $t \in [0,T]$.  Then, a direct approximation argument shows that \eqref{eq:probaCond} holds also for smooth test functions which are not compactly supported in $\Omega$, and we may choose $\psi\equiv 1$ to obtain
\begin{equation*}\label{eq:probaCondPsi1}
|\Gamma_{0-}| -  |\Gamma_t  |
   = \;  \eta\Pb( \tau^1 \le t) - \Pb( \tau^2 \le t).
\end{equation*} 
The above identity relates the change in volume of the solid region on the left-hand side to the exit probabilities of liquid and solid particles from their respective regions on the right-hand side, and can be interpreted as energy conservation.
\end{remark}

\begin{remark}\label{rem:initConstants}
Suppose that $(u^1,u^2,\Gamma)$ is a classical solution of   \eqref{eq:2SPHeat} $\!-$\eqref{eq:2SPDirichlet} with initial temperature functions 
$u_0^i$ such that $\lVert u_0^i\rVert_{L^1(\Gamma_{0-}^i)} = c_i \in (0,\infty)$, $i=1,2$. After a simple adaptation of the proof of \cref{prop:GrowthCondNoTension}, the  growth condition \eqref{eq:probaCond} becomes 
\begin{equation}\label{eq:probaCond2General}
\int_{\Gamma_{0-}}\psi -  \int_{\Gamma_t  } \psi 
   = \frac{1}{L}\big(\eta \ c_1\E^{\mu^1}[  \psi(X_{\tau^1}^1)\,\mathds{1}_{\{ \tau^1 \le t\}}] - c_2\E^{\mu^2}[  \psi(X_{\tau^2}^2)\,\mathds{1}_{\{ \tau^2 \le t\}}]\big), \quad \psi \in \mathcal{C}_c^{\infty}(\Omega).
\end{equation}
The additional flexibility in the initial temperature  when $c_i\ne 1$ proves useful in the numerical experiments (see  Sections \ref{sec:longterm} and  \ref{sec:3D2PTension}). 
\end{remark}
\begin{example}\label{ex:OnePhase}
The \textit{one-phase Stefan problem} consists of setting $u^2(t,x) \equiv 0$ in the solid region. If the  liquid temperature is initially positive, then the solid is necessarily melting, i.e.: $\Gamma_t \subseteq \Gamma_s \subseteq \Gamma_{0-}$ for all $0\le s \le t \le T$.  Writing $X = X^1$, $\tau = \tau^1$, and $\mu = \mu^1$,   the growth condition \eqref{eq:probaCond} simply reads
\begin{equation}\label{eq:probaCond1Phase}
\int_{\Gamma_{0-}\setminus \Gamma_t}\psi
   = \frac{\eta}{L}\E^{\mu}[  \psi(X_{\tau})\,\mathds{1}_{\{ \tau \le t\}}], \quad \psi \in \mathcal{C}_c^{\infty}(\Omega).
\end{equation}
\end{example}
\subsection{Adding  Surface Tension}

Consider the Stefan problem  $\eqref{eq:2SPHeat}-\eqref{eq:2SPDirichlet}$ with the Dirichlet boundary condition \eqref{eq:2SPDirichlet} replaced by 
\begin{equation}\label{eq:GibbsThomson}
    u^1  = u^2 =  - \gamma \  \kappa_{\partial \Gamma}, \quad  \text{on } \  \partial \Gamma.  \tag{2e'}
\end{equation}
The term $\kappa_{\partial \Gamma_t}$ is the mean curvature of $\partial\Gamma_t$ and $\gamma >0$ the surface tension coefficient. We use the convention that $\kappa_{\partial \Gamma_t}(x)$ is nonnegative when $\Gamma_t$  is locally convex at $x\in \partial \Gamma_t$.      
Equation \eqref{eq:GibbsThomson} captures the so-called \textit{Gibbs-Thomson}  effect  which postulates that the temperature at the interface is negative for strictly convex boundaries \cite{vis1989,luc}. In other words, the freezing point of the liquid decreases at points of convexity and increases at points of concavity. Clearly, surface tension effects only appear if $d\ge 2$. 
We now revisit  \cref{def:probSolNoTens} in the presence of surface tension. 


\begin{definition} \label{def:probSolTension}
We say that $(\mu^1,\mu^2,\Gamma)$ is a probabilistic solution of the Stefan problem \eqref{eq:2SPHeat} $\!-$\eqref{eq:GibbsThomson} with surface tension if for all $t \in [0,T]$, one has

\begin{align}\label{eq:probaCondTension}
 &\int_{\Gamma_{0-}}\psi -  \int_{\Gamma_t  } \psi 
   =  \frac{1}{L}\big(\eta \E^{\mu^1}[  \psi(X_{\tau^1}^1)\,\mathds{1}_{\{ \tau^1 \le t\}}] - \E^{\mu^2}[  \psi(X_{\tau^2}^2)\,\mathds{1}_{\{ \tau^2 \le t\}}] + \calK_{t} \big), \quad \text{where} \\[0.8em] 
&\calK_{t} = \gamma \int_0^t\int_{\partial \Gamma_s}  \left(\frac{\alpha_2}{2}\partial_{\nu} K^2 - \frac{\alpha_1}{2}\partial_{\nu} K^1\right) \psi \  , \label{eq:curvTerm} \\[1em] 
&K^i(t,x) = \E\big[ \kappa_{\! \partial \Gamma_{\overleftarrow{\; \tau}^i_{t}}}(X^i_{\overleftarrow{\; \tau}^i_{t}})\,\mathds{1}_{\{\overleftarrow{\; \tau}^i_{t} \le t\}}\ \big|\ X^i_0=x\big], \label{eq:curvK}
\end{align} 
with the backward exit times $\overleftarrow{\; \tau}^i_{t} = \inf \{s\in[0,t]:\,X_0^i+\sqrt{\alpha_i}\ W^i_s+l^i_s \notin \Gamma^{i}_{t-s} \}$. 
\end{definition} 
\begin{proposition}\label{prop:GrowthCondTension}
Suppose that  $(u^1,u^2,\Gamma)$ is a classical solution of the Stefan problem with surface tension \eqref{eq:2SPHeat} $\!-$\eqref{eq:GibbsThomson}. If  $\mu^i := \Pb \circ (X^i)^{-1}$ with $X^i$ as in \eqref{eq:rBM}, then  $( \mu^1,\mu^2 ,\Gamma)$ is a probabilistic solution of \eqref{eq:2SPHeat}$\!-$\eqref{eq:GibbsThomson}. 
\end{proposition}

\begin{proof}
See \cref{app:GrowthCondTension}.
\end{proof}

\section{Deep Level-set Method} \label{sec:Method}
Let $\Gamma = (\Gamma_t)_{t\in [0,T]}$ be the time-varying region occupied by the solid, and let $\Phi\!:[0,T]\times \Omega \to \R$ be
a \textit{level-set function} such that  
$$  \Gamma_t = \{x \in \Omega : \, \Phi(t,x) \le 0\} $$ 
and the free boundary $\partial \Gamma_t$ is given by the zero isocontour $\{\Phi(t,\cdot) = 0\}$. We also assume that the initial region $\Gamma_{0-}$ is encoded as the zero sublevel set of some function $\Phi_0\!: \Omega \to \R$.  
Whenever $\Phi \in \calC^{1,2}([0,T]\times \Omega)$ and $\nabla_x \Phi\neq 0$, the normal velocity, outward normal vector, and mean curvature of $\partial\Gamma$ read\footnote{The factor $\frac{1}{d-1}$ in the mean curvature  ensures that $\kappa_{\partial B}(x) = \frac{1}{|x|}$ for any ball $B$ around $0 \in \R^d$.} 
$$V = - \frac{\partial_t \Phi}{|\nabla_x \Phi|}, \quad \nu = \frac{\nabla_x \Phi}{|\nabla_x \Phi|}, \quad \kappa_{\partial \Gamma} = \frac{\text{div}(\nu)}{d-1}. $$
In particular, the Stefan condition \eqref{eq:2SPStefan} can be rewritten as 
 \begin{equation*}
     \partial_t \Phi = \frac{\alpha_1}{2L}  \nabla_x \Phi \cdot \nabla_x u^1 - \frac{\alpha_2}{2L} \nabla_x \Phi \cdot \nabla_x u^2. 
 \end{equation*}   
The \textit{signed distance function} defined by
$$\rho(t,x) = (-1)^{\mathds{1}_{\Gamma_{t}}}\ dist(x,\partial \Gamma_{t})
$$
can always be used as a level-set function and encodes all geometric properties of the interface; see \cite{AS,Soner} and \cite[Section I.2]{OsherFedkiw}. In particular, suppose that $\partial \Gamma_t$ is a smooth manifold. Then, $\rho$ is differentiable in a neighborhood of $\partial \Gamma_t$ and satisfies
$$
|\nabla_x \rho| =1, \quad 
V = -\partial_t \rho,\quad
\nu = \nabla_x \rho, \quad
\kappa_{\partial \Gamma} = \frac{\Delta_x \rho}{d-1},\qquad \text{on }\ \partial\Gamma_t.
$$

To approximate a general level-set function $\Phi$,  we parameterize the difference $\Phi(t,\cdot) - \Phi_0$ by a neural network $G\!:[0,T]\times \Omega \times \Theta \to \R$, for some   parameter set $\Theta\subset \R^p$, $p\in \N$. This leads to  the \textit{deep level-set function}  
\begin{equation}\label{eq:deepLevelSet}
    \Phi(t,x;\theta) := \Phi_0(x) + G(t,x;\theta) , \quad t\in [0,T], 
\end{equation} 
with the associated  regions $ \Gamma^{1,\theta} = \{\Phi(\cdot \ ;\theta) > 0\}$ and  $\Gamma^{\theta} =\Gamma^{2,\theta} = \{\Phi(\cdot \ ;\theta) \le  0\}$. By parameterizing the difference we make sure that the initial condition holds automatically. It also allows us to capture the occasional jumps in the solid region at $t=0$ when the initial data $(u_0^1,u_0^2)$ exhibits discontinuities; see \cref{sec:jumpExample}. 

\begin{example}\label{ex:radial}
(Radial case) Let $\Omega = B_R := \{|x|\le R\}$ and $\Gamma_{0-} =  B_{r_0}$ for some $r_0 \in (0,R)$. We can then initialize the level-set function with the signed distance  $\Phi_0(x) = |x| - r_0$. If the initial temperature inside the solid and the liquid is radially symmetric, then $\Gamma_t =  B_{r(t)}$ for some càdlàg function $r\!:[0,T] \to [0,R]$. We can therefore set $\Phi(t,x) = |x| - r(t)$, as illustrated in \cref{fig:levelSetViz}. This implies that
$\Phi(t,x) - \Phi_0(x) = r_0 - r(t)$ for all $x \in \Omega.$ 
Hence, the  neural network, although unaware of the radial symmetry, ``simply" needs to learn a function of time. Indeed, in our numerical experiments, we do not impose radial symmetry on the solution, but rather let the neural network learn this invariance through training.
\end{example}

\begin{figure}[t]
    \centering
    \caption{Time evolution of the level-set function in the radial case and projected solid region (bottom horizontal plane).}
    \includegraphics[height = 1.9in,width = 5.8in]{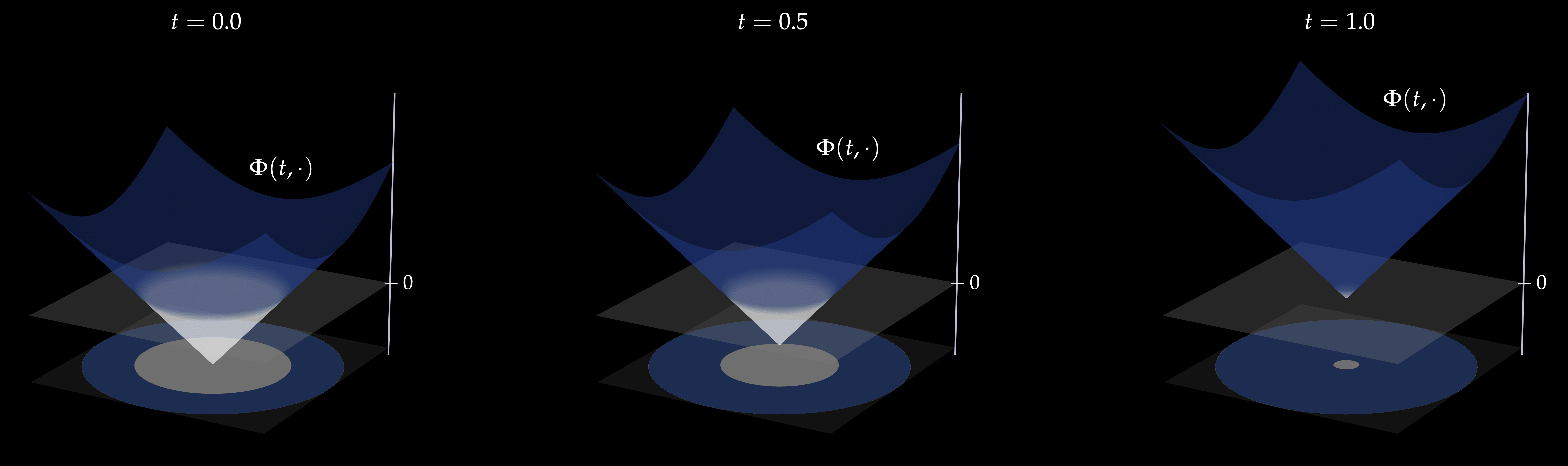}
\label{fig:levelSetViz}
\end{figure}

\subsection{Loss Function and Training}
Let us explain how the parameters in $\Gamma^{\theta}$ can be trained so as to find a probabilistic solution, which is equivalent to requiring the growth condition \eqref{eq:probaCond} to hold for every test function. This is similar to Leray type weak solutions of nonlinear partial differential equations that demand certain equations to be satisfied for a class of test functions, and our approach can possibly be applied to compute this type of solutions as well. We transform the growth condition into a loss function by forcing \eqref{eq:probaCond} to hold for a large but finite number of test functions. Specifically, letting $\Psi \subset \calC_c^{\infty}(\Omega)$ be a finite collection of test functions, we aim to
\begin{equation}\label{eq:minimization}
    \text{minimize }\; \calL(\theta) = \sum_{\psi\in \Psi} \calL(\theta, \psi) \; \text{ over } \; \theta \in \Theta, \quad \text{where}
\end{equation}
\begin{align}\label{eq:loss}
    \calL(\theta,\psi) =  \bigg \lVert  \int_{\Gamma_{0-}} \psi - \int_{\Gamma^{\theta}_{\cdot}} \psi   -  \frac{1}{L} \big( \eta\  \mathbb{E}[\psi(X^1_{\tau^{1,\theta}})\,\mathds{1}_{\{\tau^{1,\theta} \  \le \ \cdot\}}] -  \mathbb{E}[\psi(X^2_{\tau^{2,\theta}})\,\mathds{1}_{\{\tau^{2,\theta} \  \le \ \cdot\}}] \big) \bigg \rVert_{L^2([0,T])}^2. 
\end{align} 
We expect that $\Gamma^{\theta}$ induces a probabilistic solution if $\calL(\theta)=0 $ and  $\Psi$ grows to a dense subset of $\calC_c^{\infty}(\Omega)$. In order to compute the loss function in \eqref{eq:minimization},  we simulate reflected Brownian particles $X^j = (X_{t_n}^j)_{t_n \in \Pi^N}$, $j=1,\ldots,J$, on a regular time grid $\Pi_N:=\{t_n = n\Delta t \! : n=0,\ldots, N \}$, $\Delta t = \frac{T}{N}$. 
This gives the empirical loss 
\begin{align}
   & \calL(\theta) = \sum_{\psi\in \Psi} \sum_{n=1}^N \ell_{n}(\theta,\psi)^2,  \label{eq:totalLoss}\\[0.4em]
   & {\scriptsize \ell_n(\theta,\psi)} =  \scriptsize \int_{ \Gamma_{0-}} \psi   - \int_{\Gamma^{\theta}_{t_n}} \psi   - \ \frac{1}{JL}\sum_{j=1}^J\Big(\eta \ \psi(X^{1,j}_{\tau^{1,\theta,j}})\,\mathds{1}_{\{\tau^{1,\theta,j}  \le t_n\}} - \ \psi(X^{2,j}_{\tau^{2,\theta,j}})\,\mathds{1}_{\{\tau^{2,\theta,j}  \le t_n\}}\Big) . \label{eq:partialLoss}
\end{align}

We then train the deep level-set function by gradient descent. That is, the parameters are updated according to 
\begin{equation}\label{eq:SGD}
    \theta_{m+1} = \theta_{m} - \zeta_m \nabla_\theta \calL(\theta_m), \quad  m = 1, \ldots, M,
\end{equation}
for some $M\in \N$. However, the map $\theta \mapsto \tau^{\theta}$ is piecewise constant, where $\tau^{\theta}$ denotes a  stopping time on the right-hand side of \eqref{eq:partialLoss}. For concreteness, let us suppose that $\tau^{\theta}$ is the exit time of a liquid particle from the liquid region. In terms of the parameterized level-set function, this reads
$$\tau^{\theta} = \inf\{t_n, n\le N: \, \Phi(t_n,X_{t_n}^1;\theta) < 0\}. $$
Unless a trajectory $(X^1_{t_n}(\omega))_{n=0}^N$ is exactly on the solid-liquid interface at the stopping time $\tau^{\theta}(\omega)$, i.e., $\Phi(\tau^{\theta}(\omega),X_{\tau^{\theta}}^1(\omega);\theta)=0$,  the value of $\tau^{\theta}(\omega)$ will remain the same after an infinitesimal change in the parameter vector $\theta$. 
Hence, the gradient descent would not converge due to the vanishing gradient $\nabla_{\theta} \tau^{\theta}$. To circumvent this issue, the stopping times are relaxed according to a procedure explained in the next section. 

Regarding the test functions, we choose  the Gaussian kernels\footnote{If $\beta_k$ is large enough, then, at least numerically,  $\psi_k$ is  compactly supported in $\Omega$. }   
 $$\Psi = \left\{\psi_k(x) = e^{-\beta_k |x-z_k |^2} : \,  \beta_k >0, \ z_k\in \Omega, \  k=1,\ldots, K \right\}, \quad K\in \N.$$
The centers $z_k$ and widths $\beta_k$ are randomized in each training iteration to cover the whole domain and to capture a variety of scales. 

\subsection{Relaxed Stopping Times}

 We adapt the relaxation procedure of \cite{ReppenSonerTissotFB}, proposed in the context of optimal stopping. In short, it consists of replacing the sharp boundary $\partial \Gamma_t$ by a \textit{mushy} (or \textit{fuzzy}) \textit{region} where stopping may or may not take place.   
 First, we transform the neural network $\Phi(t,\cdot\ ;\theta)$ to locally approximate the signed distance to the interface $\partial \Gamma_t$. This is achieved by normalization, namely  by setting  $\rho(t,\cdot \ ; \theta) = \frac{\Phi(t,\cdot\ ;\theta)}{|\nabla_x \Phi(t,\cdot\ ;\theta)|}$.\footnote{Indeed, write $\phi = \Phi(t,\cdot\ ;\theta)$  and  fix $x \notin  \partial \Gamma_t$, i.e., $\phi(x) \ne 0$. If $\nu$ is the outward normal vector at the point on $\partial \Gamma_t$ closest to $x$, then the signed distance $\rho \in \R$ of $x$ to $\partial \Gamma_t$ can be characterized as the smallest solution in absolute value to $\phi(x - \rho \nu) = 0$. A Taylor approximation gives $0 = \phi(x - \rho \nu) \approx \phi(x) - \rho \ \nu \cdot \nabla \phi(x)$, so  $\rho \approx  \frac{\phi(x)}{|\nabla \phi(x)|}$ as desired. Note that this approximation is only accurate close to the interface, 
 which is precisely where $\rho$ plays a role in the algorithm.} We note that the spatial gradient of $\Phi(\cdot; \theta)$ can be exactly and effortlessly computed using automatic differentiation.  

 Next, the approximate signed distance is converted into a stopping probability. Without loss of generality, we describe the procedure for liquid particles which are stopped when entering the solid region. For solid particles, see \cref{rem:solidParticles}. Given $\varepsilon > 0$, define the phase indicator variables 
  $$q^{\theta,\varepsilon}_n = (\chi^{\varepsilon} \circ \rho)(t_n,X_{t_n}^1;\theta),  \quad n = 0,\ldots,N, $$
with the \textit{relaxed phase function} $\chi^{\varepsilon}(\rho) = \frac{(1 - \rho/\varepsilon)^+}{2} \wedge 1$; see   \cref{fig:hFunction} for an illustration. The phase indicator variable is equal to $1$ in the ``$\varepsilon-$interior" of the solid region (where $\rho(t,x;\theta) \le -\varepsilon$) and is $0$ when $\rho(t,x;\theta) \ge \varepsilon$. Inside $\partial \Gamma_t^{\theta,\varepsilon}$, the phase indicator variable decreases linearly with the signed distance to the interface. The parameter $\varepsilon>0$ specifies the width of the \textit{mushy region} $$\partial \Gamma_t^{\theta,\varepsilon} = \{x \in \Omega: \, |\rho(t,x;\theta)| < \varepsilon \},$$ in which a liquid particle enters the solid region with probability $q^{\theta,\varepsilon}_n \in (0,1)$. 

For Brownian particles, a natural choice for $\varepsilon$ is $\varepsilon = \sqrt{d \ \Delta t}$, this being the order of the typical Euclidean distance travelled by a $d$-dimensional standard Brownian motion over a time sub-interval $[t,t+\Delta t]$. For Brownian particles with diffusivity $\alpha >0$, we set instead $\varepsilon  = \sqrt{\alpha d \Delta  t}$.
When the liquid and the solid particles have different diffusivities, we consider distinct mushy regions.  
Finally, the stopped values in \eqref{eq:partialLoss} are  replaced by
$$\psi(X_{\tau^{1,\theta}}^1)\,\mathds{1}_{\{\tau^{1,\theta} \  \le \ t_n\}} \approx  \sum_{l=1}^n Q^{\theta,\varepsilon}_l \ \psi(X_{t_l}^1),$$
with the random stopping probabilities recursively given by
$$\quad Q_{n}^{\theta, \varepsilon} = q^{\theta,\varepsilon}_n\underbrace{\Big(1- \sum_{l<n} \ Q_{l}^{\theta, \varepsilon}  \Big)}_{\text{same phase until $t_n$}}, \quad n = 0,\ldots, N. $$

  Similarly, the integrals in \eqref{eq:partialLoss} are computed  as  
  \begin{equation}\label{eq:intMC}
      \int_{\Gamma_{t_n}^{\theta}} \psi \ \approx \  
  \frac{|\Omega|}{J}\sum\limits_{j=1}^{J} q_{n,0,j}^{\theta, \varepsilon}\ \psi(U^j), \quad   \  q_{n,0,j}^{\theta, \varepsilon} = (\chi^{\varepsilon} \circ \rho)(t_n,U^j;\theta), \quad  U^j \overset{\text{i.i.d.}}{\sim} \ \text{Unif}(\Omega),
  \end{equation}
where Unif$(\Omega)$ is the uniform distribution on $\Omega$. Notice that the terms $\int_{ \Gamma_{0-}} \psi_k $ do not depend on~$\theta$ so they can  be computed only once, e.g., in an offline phase. Because of the above relaxation, the gradient of $\calL(\theta)$ does not vanish anymore. We can therefore apply gradient descent to minimize the loss function, thus finding an approximation of the solid region $\Gamma$.  
\begin{figure}[t]
     \centering
     \caption{Illustration of the relaxed phase function $\chi^{\varepsilon}(\rho) = \frac{(1 - \rho/\varepsilon)^+}{2} \wedge 1$.}
     \includegraphics[height = 1.8in,width = 2.5in]{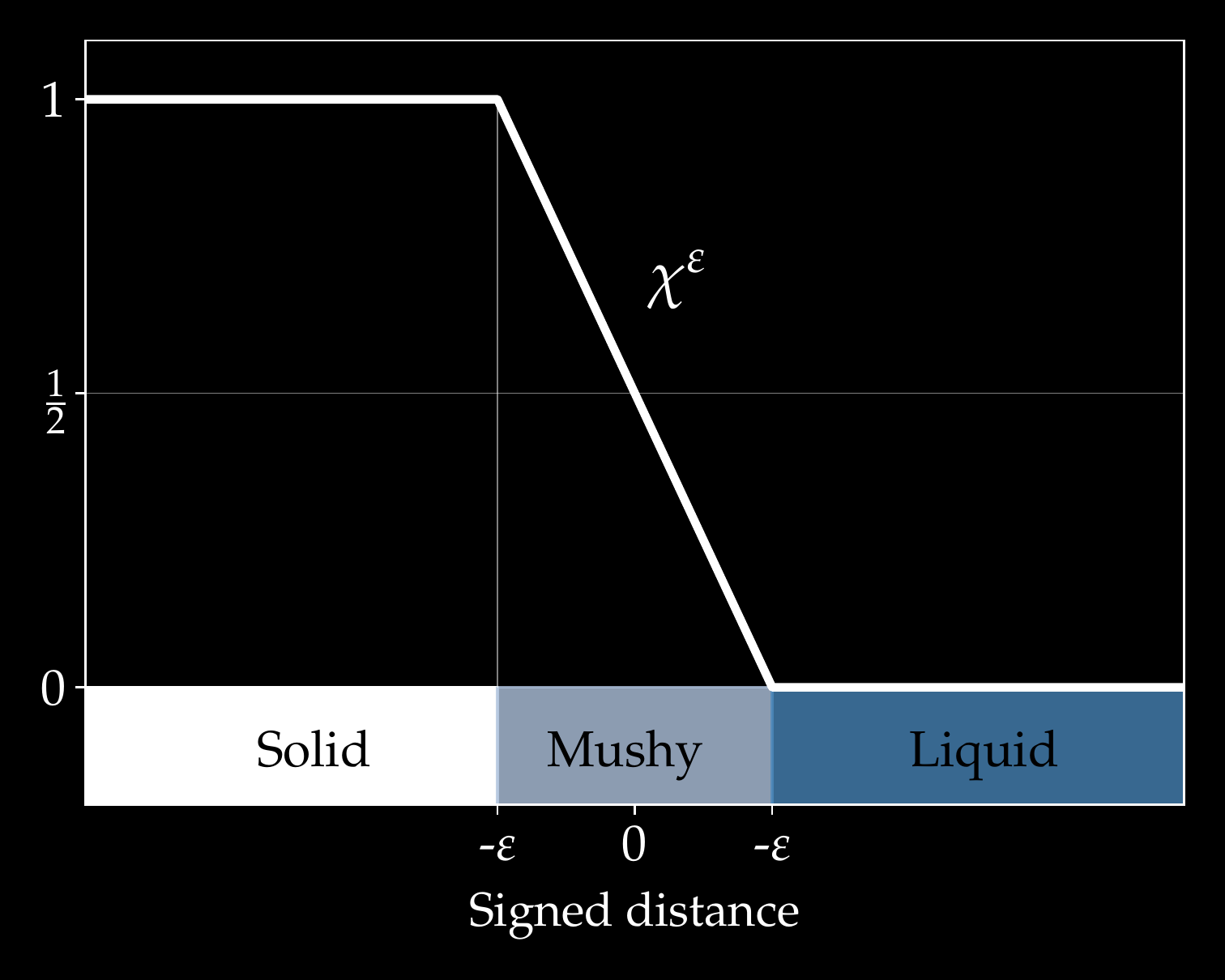}
     \label{fig:hFunction}
 \end{figure}
 

\begin{remark}\label{rem:solidParticles}
For solid particles,  the phase indicator variables are defined analogously by $$q^{\theta,\varepsilon}_n = 1 - (\chi^{\varepsilon} \circ \rho)(t_n,X_{t_n}^2;\theta).$$ We can thereafter follow the  steps above to approximate the stopped values $\psi(X_{\tau^{2,\theta}}^2)\,\mathds{1}_{\{\tau^{2, \theta} \  \le \ t_n\}}$.
\end{remark}

\subsection{Jump penalty}\label{sec:jumpPenalty}
So far, there is no control on the change in volume of the solid region. Although jumps can be observed \cite{DNS}, we may find a solution which satisfies the growth condition but exhibits non-physical jumps, i.e., jumps of non-physically large size. To deal with this issue, we consider the more general loss function 
$\underline{\calL}(\theta) = \calL(\theta) + \lambda \calP(\theta)$, with the penalty term

\begin{equation}\label{eq:penalty}
    \calP(\theta) =  \calR\Big( \max\limits_{n=1,\ldots,N} |\Gamma^{\theta}_{t_n} \triangle \Gamma^{\theta}_{t_{n-1}}| - C \Big)  , \quad C >0, \quad \calR(\Delta) = \Delta^+ - 0.01\Delta^-, 
\end{equation}
where $\triangle$ is the symmetric difference operation, and $C > 0$ is a constant describing the ``allowed" volume increment between time steps. We note that $\calR$ is the so-called \textit{Leaky ReLU} activation function, whose choice is justified as follows.  If the maximum change in volume exceeds the threshold $C$,  the deep level-set function is  severely penalized by the ``ReLU" part of $\calR$ (namely, $\Delta^+$). On the other hand, the ``leaky" part of $\calR$, i.e., $0.01\Delta^-$, is introduced to slightly incentivize the solid to shrink/grow smoothly over time. 

 
 The Lagrange multiplier $\lambda$ is decomposed into  the product of a fixed ratio $\lambda_0 > 0$  (typically less than one) and a dynamically updated factor $\lambda^{\text{\tiny  scale}}$  such that  $\lambda^{\text{\tiny scale}} \ \nabla_{\theta} \calP(\theta)$ 
 matches the scale of  the gradient $\nabla_{\theta}\calL(\theta)$. For this task, we employ the  ``learning rate annealing algorithm"  outlined in  \citet{WangLearningRate}. It is a popular normalization technique for physics-informed neural networks when the loss function  comprises heterogeneous components.  Finally, the volume of $\Gamma^{\theta}_{t_n} \triangle \ \Gamma^{\theta}_{t_{n-1}}$ is estimated using a Monte Carlo simulation as in \eqref{eq:intMC}. 

  Throughout the numerical experiments, we set the jump  threshold to $C = \frac{|\Omega|}{2}$. In other words, the solid region is told to grow/shrink by at most half the size of the domain per time step. This is a loose constraint, which nevertheless rules out undesirable time discontinuities as we shall see in the numerical experiments. 
 At the same time, we shall see in \cref{sec:jumpExample} that the regularized algorithm is still capable of producing 
 \textit{physical jumps}, whose definition is now recalled. We also refer the interested reader to \cite{MishaRadialTension}. 
Let us focus on the radial case and consider for $r\ge 0$, $\Delta \in \R$, the annulus
\begin{equation}\label{eq:annuli}
    A_{r,\Delta } := \begin{cases}B_{r+\Delta} \setminus B_{r}, & \Delta \ge 0,\\
    B_{r}  \setminus B_{r+\Delta}, & -r \le \Delta < 0,\\
    B_r, & \Delta < -r. 
    \end{cases} 
\end{equation}
A radial solution $(\Gamma_t)_{t\in [0,T]} = (B_{r(t)})_{t\in [0,T]}$ to the Stefan problem is \textit{physical} if for all $t\in [0,T]$, the jump $\Delta r (t) = r(t) - r(t-)$, if positive, satisfies 
\begin{align}\label{eq:physicalJump}
\Delta r(t) = 
\inf\bigg\{\Delta > 0 \ : \  |A_{r(t-),\Delta }| > -\frac{1}{L}\int_{A_{r(t-),\Delta  }}   u^1(t-,x)\bigg\}. 
\end{align}
If $\Delta r(t)=r(t)-r(t-)>0$, we gather that for physical solutions, $\Delta r(t)$ is the smallest solution of 
\begin{align}
|A_{r(t-),\Delta }|&=  -\frac{1}{L}\int_{A_{r(t-),\Delta }} u^1(t-,x). 
 \label{eq:eqPosJump} 
\end{align}
The left-hand side in \eqref{eq:eqPosJump} is the volume absorbed by the solid region $\Gamma$ at time $t$. The right-hand side in \eqref{eq:eqPosJump} is the aggregate change in temperature upon freezing of $A_{r(t-),\Delta }$.

\begin{remark}
In the non-radial case, it is not sufficient to consider the volume increments only. The formulation of an appropriate physicality condition is in fact subject of ongoing research.
\end{remark}


 \subsection{Algorithm} 

The deep level-set method is summarized in \cref{alg:DLSM1}. The learning rate process $(\zeta_{m})_{m=1}^M$ is computed using the Adam optimizer \cite{Kingma}. For the Monte Carlo simulation of particles, we use antithetic sampling  to explore the domain in a symmetric fashion and to speed up computation. More specifically, we simulate two antithetic particles  for each generated initial point $x_0^i$, namely  
\begin{align*}
    X^i_\cdot &= x_0^i + \sqrt{\alpha_i} \ W_{\cdot\wedge \tau^i}^i + l_{\cdot\wedge \tau^i}^i, \quad 
    \tau^i = \inf\{t\in [0,T]:\,x_0^i + \sqrt{\alpha_i} \ W_t^i + l^i_t \notin \Gamma_t^i \},\\[0.5em]
    \widetilde{X}^i_\cdot &= x_0^i - \sqrt{\alpha_i} \ W_{\cdot\wedge \widetilde{\tau}^i}^i + \widetilde{l}_{\cdot\wedge \widetilde{\tau}^i}^i, \quad 
   \widetilde{\tau}^i = \inf\{t\in [0,T]:\,  
   x_0^i - \sqrt{\alpha_i} \ W_t^i + \widetilde{l}_t^i
   \notin \Gamma_t^i \}. 
\end{align*}

\begin{algorithm}[H]
\caption{(Deep level-set method) }\label{alg:DLSM1} 
\vspace{1mm}
\textbf{Given}: $\Phi_0=$ initial level-set  function, $J=$ batch size, $M=$ \# training iterations, $\Psi = $ set of test functions, $C=$ jump constant, $\varepsilon=$ mushy region width\\[-0.8em] \noindent\rule{\textwidth}{0.5pt} 
\begin{itemize}
 \setlength \itemsep{-0.1em}
 
\item[I.] \textbf{Initialize} $\theta_0 \in \Theta$
\item[II.] For $m = 0,\ldots, M-1$:
    \begin{enumerate}
    \setlength \itemsep{0.5em}

  \item \text{\textbf{Simulate} $X^{i,j} = (X_{t_n}^{i,j})_{n=0}^N$ ($i=1$: liquid, $i=2$: solid),  $\ U^j  \ \overset{\text{i.i.d.}}{\sim} \  \text{Unif}(\Omega)$, $j=1,\ldots,J$}
  
  \item \textbf{Approximate} the signed distance function of $\partial\Gamma$ as $\rho(\cdot;\theta_{m}) = \frac{ \Phi(\cdot;\ \theta_{m})}{|\nabla_x \Phi(\cdot;\ \theta_{m})|}$
  
        \item For $n=0,\ldots,N$, $\psi \in \Psi$, \textbf{compute}: 
        \begin{itemize}
         \item \textbf{Phase indicator variables:} 
        \text{$q_{n,i,j}^{\theta_{m},\varepsilon} = \begin{cases}
        (\chi^{\varepsilon} \circ \rho)(t_n,U^j;\theta_{m}), & i=0,\\[0.2em]
             (\chi^{\varepsilon} \circ \rho)(t_n,X_{t_n}^{1,j};\theta_{m}), & i=1,\\[0.2em]
            1-  (\chi^{\varepsilon} \circ \rho)(t_n,X_{t_n}^{2,j};\theta_{m}), & i=2 \\
        \end{cases}$}
       \item   \textbf{Symmetric Differences:} $|\Gamma^{\theta_m}_{t_n} \triangle \Gamma^{\theta_m}_{t_{n-1}}|  \approx \frac{|\Omega|}{J}\sum\limits_{j=1}^{J} |q_{n,0,j}^{\theta_m, \varepsilon} -  q_{n-1,0,j}^{\theta_m, \varepsilon}|  $\\[-0.2em] 
                 \item   \textbf{Integrals:} $\int_{\Gamma^{\theta_m}_{t_n}} \psi \ \approx \  \frac{|\Omega|}{J}\sum\limits_{j=1}^{J} q_{n,0,j}^{\theta_m, \varepsilon}\ \psi(U^j)$\\
                 [-0.2em]
               
            \item \textbf{Stopping probabilities: $ Q_{n,i,j}^{\theta_m,\varepsilon} = q_{n,i,j}^{\theta_{m},\varepsilon} ( 1 - \sum\limits_{l<n} Q_{l,i,j}^{\theta_{m},\varepsilon} ), \; i=1,2$}   \\[-0.2em]
            
            \item \textbf{Stopped values: } \text{$\psi(X^{i,j}_{\tau^{\theta_m}})\,\mathds{1}_{\{\tau^{\theta_m}   \le t_n\}} \approx \sum\limits_{l=1}^n Q_{l,i,j}^{\theta_m,\varepsilon}\psi(X_{t_l}^{i,j})$}\\[-0.4em]
  
        \end{itemize}
\item  \textbf{Loss}: $\calL(\theta_m) = \sum\limits_{\psi\in \Psi}\sum\limits_{n=1}^N\ell_n(\theta_m,\psi)^2$, with 
\begin{align*}
    {\scriptsize \ell_n(\theta_m,\psi)} &=   {\scriptsize \int_{ \Gamma_{0-}} \! \psi   - \int_{\Gamma^{\theta_m}_{t_n}} \! \psi}   \\
    &- {\scriptsize \ \frac{1}{JL}\sum_{j=1}^J\Big(\eta \ \psi(X^{1,j}_{\tau^{1,\theta_m,j}})\,\mathds{1}_{\{\tau^{1,\theta_m,j}  \le t_n\}} - \ \psi(X^{2,j}_{\tau^{2,\theta_m,j}})\,\mathds{1}_{\{\tau^{2,\theta_m,j}  \le t_n\}}\Big)} 
\end{align*}

\item \textbf{Jump penalty:} \text{ $\ \calP(\theta_m) =  \calR \Big(\max\limits_{n=1,\ldots,N} |\Gamma^{\theta_m}_{t_n} \triangle \Gamma^{\theta_m}_{t_{n-1}}| - C \Big)$} \\[0.1em]
\item \textbf{Gradient step:} $\theta_{m+1} = \theta_{m} - \zeta_m \nabla_\theta \underline{\calL}(\theta_m) $, $\; \underline{\calL}(\theta_m) = \calL(\theta_m) + \lambda_0 \ \lambda_m^{\text{\tiny scale}} \calP(\theta_m) $
   \end{enumerate}
\item[III.] \textbf{Return} $\Gamma^{\theta_M} = \{\Phi(\cdot \ ; \theta_{M}) \le 0\}$\\[-1.5em]
\end{itemize}
\end{algorithm}

\subsection{Adding surface tension} \label{sec:MethodTension}

\subsubsection{Growth Condition Revisited }
The numerical verification of the growth condition given in \cref{def:probSolTension} is not straightforward because of the additional term $\calK_t$. It turns out that  an  alternative formulation, which we now outline, greatly simplifies the implementation. In what follows, we assume  that $d\ge 2$. 
For $\delta > 0$, define the one-sided mushy regions,
$$\partial \Gamma_t^{\delta, i}  = \{x\in \Gamma_t^i : \, dist(x,\partial \Gamma_t)\le \delta \}, \quad i=1,2. $$ 

Let $(T_l^{\delta,i})_{l\ge 1}$ be the arrival times of  a time-space Poisson point process with intensity
\begin{equation}\label{eq:PoissonIntense}
     \lambda^i(t,y) = \frac{\gamma}{\delta} |\kappa_{\partial\Gamma_t}(\widehat{y})|^{2-d} \frac{\mathds{1}_{\partial \Gamma_t^{\delta,i}}(y)}{|\partial \Gamma_t^{\delta,i}|},\quad \widehat{y} = \text{proj}_{\partial \Gamma_t}(y), \quad d\ge 2,
\end{equation}
where $\text{proj}_{\partial \Gamma_t}(y)$ is the projection of $y$ onto $\partial \Gamma_t$, i.e., the point on $\partial \Gamma_t$ closest to $y$. For each $l\ge 1$ and $i=1,2$, we simulate a Brownian particle $Y^{l,i} = (Y^{l,i}_s)_{s\in[T_l^{\delta,i},T]}$ where, according to \eqref{eq:PoissonIntense}, the initial position of $Y^{l,i}$ is uniformly distributed in $\partial \Gamma_{T^{\delta,i}_l}^{\delta,i}$ if $d= 2$ and inversely proportional to the absolute mean curvature of $\partial \Gamma_{T^{\delta,i}_l}$ otherwise.  
We also define the exit times $\tau_l^{\delta,i} = \inf\{s \ge T^{\delta,i}_l \! : Y^{l,i}_s \notin \Gamma_s^i \}$. 

In the three-dimensional, radially symmetric case (namely $\Gamma_t = B_{r(t)} \subseteq \R^3$), the Poisson intensities simply read
\begin{equation}\label{eq:PoissonIntenseRad}
     {\lambda^i(t,\cdot) = \frac{\gamma}{\delta}} r(t) \frac{\mathds{1}_{A_{r(t),\Delta_i}}}{|A_{r(t),\Delta_i}|}, \quad   \Delta_i = (-1)^{i+1} \delta,
\end{equation}
with the annuli $A_{r,\Delta}$ defined in \eqref{eq:annuli}. We therefore gather that the expected number of simulated particles is directly proportional to the radius of the solid region. 


\begin{proposition}\label{prop:growthCondTensionPoisson}
Assume that $\alpha_1 = \alpha_2$ and consider a classical solution of the radially symmetric Stefan problem with surface tension. Moreover, let $N_t^{\delta,i} := \sum_{l=1}^{\infty} \mathds{1}_{\{T^{\delta,i}_l \le t\}}$ be the counting process associated with $(T^{\delta,i}_l)$, $i=1,2$.  Then for all $ \psi \in \mathcal{C}_c^{\infty}(\Omega)$, 
\begin{align}
   &   \int_{\Gamma_{0-}}\psi -  \int_{\Gamma_t  } \psi
   =  \frac{1}{L} \Big( \eta \E^{\mu^1}[ \psi(X_{\tau^1}^1)\,\mathds{1}_{\{ \tau^1 \le t\}}] - \E^{\mu^2}[  \psi(X_{\tau^2}^2)\,\mathds{1}_{\{ \tau^2 \le t\}}] + \calK_{t}   \Big),\label{eq:probacondTension2}  \\[1em]  
&  \calK_{t} = - \calK_t^1 - \calK_t^2, \\[1em]
&  \calK_t^i = \lim_{\delta \downarrow 0} \  \E\Bigg[\sum_{l = 1}^{N_t^{\delta,i} } \Big( \psi(Y_{\tau_l^{\delta,i}}^{l,i}) \,\mathds{1}_{\{\tau^{\delta, i}_l \le t\}} \ {- \  \psi(Y_{T^{\delta,i}_l}^{l,i})}   \Big)\Bigg]. \label{eq:curvatureTerm}
\end{align}
\end{proposition}

\begin{proof}
See \cref{app:growthCondTensionPoisson}. 
\end{proof}

\begin{remark}
    When the diffusivities of the solid and the liquid phase are different ($\alpha_1 \ne \alpha_2$), the right-hand side of the growth condition \eqref{eq:probacondTension2} contains an additional term, namely 
$$\widetilde{\calK}_t  =  - \gamma\frac{\alpha_2-\alpha_1}{2} \int_0^t\int_{\partial \Gamma_s} \kappa_{\partial \Gamma_s} \  \partial_{\nu} \psi.$$ 
Recalling that $u^1=u^2 = - \gamma \kappa_{\partial \Gamma_s} $, the above term can be approximated via
$$\widetilde{\calK}_t \approx    \frac{\alpha_2}{2\delta } \int_0^t\int_{\partial \Gamma_s^{\delta,2}} u^2 \  \partial_{\nu} \psi - \frac{\alpha_1}{2\delta } \int_0^t\int_{\partial \Gamma_s^{\delta,1}} u^1 \  \partial_{\nu} \psi.$$ 
The integrals $\int_{\partial \Gamma_s^{\delta,i}} u^i \  \partial_{\nu} \psi$, in turn, can be computed using the stochastic representation of $u^i$ given in \eqref{eq:intu^i}, below. Note that $\partial_{\nu} \psi$ is easy to compute since the outward normal $\nu$ is available from the deep level-set function. 
\end{remark}

\begin{remark} \label{rem:curvconjecture}
It is conjectured that the representation \eqref{eq:curvatureTerm} of the curvature terms holds beyond the radial case. The only difference is that the effect of boundary particles is reversed in \textit{concave} sections of the boundary. Specifically, we expect that, in general,
\begin{equation}
    \calK_t^i = \lim_{\delta \downarrow 0} \  \E\Bigg[\sum_{l = 1}^{N_t^{\delta,i} } \text{sign}\big(\kappa_{\partial \Gamma_{ \! T_l^{\delta,i}}}(\widehat{Y_{T^{\delta,i}_l}^{l,i}})\big)\big( \psi(Y_{\tau_l^{\delta,i}}^{l,i})\,\mathds{1}_{\{\tau_l^{\delta,i} \le t\}} \ - \  \psi(Y_{T^{\delta,i}_l}^{l,i})   \big)\Bigg].\label{eq:curvatureTerm2}
\end{equation}
\end{remark}
 We show how to approximate $\calK_t^i$ of \eqref{eq:curvatureTerm2}. First, the mushy regions can be expressed in terms of  the level-set function as follows:  $$\partial \Gamma_t^{\delta,1} \approx \left \{x\in \Omega :\, \frac{\Phi(t,x)}{|\nabla_x \Phi(t,x)|} \in (0,\delta] \right\}, \quad \partial \Gamma_t^{\delta,2} \approx \left \{x\in \Omega  : \, \frac{\Phi(t,x)}{|\nabla_x \Phi(t,x)|} \in [-\delta,0] \right\}.$$ 
Without loss of generality, we focus on the term $\calK_t^1$ and drop the superscripts $^1$ throughout. 
Using the time discretization  $\{t_n = n\Delta t \! :  n=0,\ldots,N\}$, $\Delta t = \frac{T}{N}$, $N\in \N$, we obtain for $\delta>0$ sufficiently small that   
$$\calK_{t_n} \approx \ \sum_{m=0}^{n-1} \E\Bigg[\sum_{l =N^{\delta}_{t_{m}}{+1}}^{N^{\delta}_{t_{m+1}}}  \E\Big [\text{sign}\big(\kappa_{\partial \Gamma_{ \! t_m}}(\widehat{Y_{t_m}^l})\big)\big(\psi(Y_{\tau_{l}^{\delta}}^{l})\,\mathds{1}_{\{\tau_{l}^{\delta} \le t_n\}}  - \psi(Y_{t_m}^{l})\big) \ \Big| \ T^{\delta}_l = t_m \Big]  \Bigg].$$ 
Writing $\E_{t_m,y}[\cdot ] = \E[ \ \cdot \ | \ T^{\delta}_l = t_m,\ Y^l_{t_m}  = y]$ and recalling the Poisson intensity in \eqref{eq:PoissonIntense} we find that 
\begin{align}
 & \,\calK_{t_n} \approx \ \sum_{m=0}^{n-1} \frac{\gamma \Delta t }{\delta\ |\partial \Gamma_{t_m}^{\delta}|} \int_{\partial \Gamma_{t_m}^{\delta}} |\kappa_{\partial\Gamma_{t_m}}(\widehat{y})|^{2-d} \E_{t_m,y}\big[\text{sign}\big(\kappa_{\partial \Gamma_{ \! t_m}}(\widehat{y})\big)\big(\psi(Y_{\tau^{\delta}})\,\mathds{1}_{\{\tau^{\delta} \le t_n\}}  - \psi(y)\big)\big]\,\mathrm{d}y \nonumber \\
& = \ \frac{\gamma \Delta t }{\delta}   \sum_{m=0}^{n-1}  \E\Big[\ \text{sign}\big(\kappa_{\partial\Gamma_{t_m}}(\widehat{Y_{t_m}})\big)\big|\kappa_{\partial\Gamma_{t_m}}(\widehat{Y_{t_m}})\big|^{2-d}\E\big[\psi(Y_{\tau^{\delta}})\, \mathds{1}_{\{\tau^{\delta} \le t_n\}}{ - \psi(Y_{t_m})} \ \big| \  T^{\delta} = t_m,\ Y_{t_m}\big]\Big]. \label{eq:curvCompute}
\end{align}
The latter is estimated through Monte Carlo simulation where, for simplicity, the curvature is evaluated at $Y_{t_m}$ instead of its projection $\widehat{Y_{t_m}}$ onto $\partial\Gamma_{t_m}$. This is an accurate approximation of the ``true'' curvature when the width $\delta$ of the mushy region is small. What remains is to approximate the mean curvature, which we address in the next section. 

\subsubsection{Mean Curvature  Approximation}\label{sec:trickradial}
In this section, we present simple algorithms to locally approximate the mean curvature of a manifold $\Sigma  \subset \R^d$, $d=2,3$, of codimension $1$. Other techniques, e.g., using finite differences, are discussed in \cite{Popinet}. If $\Sigma = \{\phi = 0\}$ for some level-set function $\phi \in \calC^2(\R^d)$, the mean curvature  can be expressed as
\begin{equation}
    \kappa_{\Sigma}(y) = \frac{\text{div}\left(\nu(y)\right)}{d-1}, \quad y\in \Sigma, 
\end{equation} 
with the outward normal vector field $\nu = \frac{\nabla \phi}{|\nabla \phi|} $. If $\phi$ is parameterized by a neural network, one could apply automatic differentiation to compute $\kappa_{\Sigma}$. But computing second order  derivatives through automatic differentiation is  costly since it entails nested gradient tapes. Indeed, the machine would need to  keep track of the first order derivatives as well. Alternative approaches have been proposed to speed up computation,  e.g., using Monte Carlo techniques \cite{DGM}.  
Here, we propose a dilation approach echoing the geometric nature of curvature and solely exploiting  the gradient of $\phi$. 

\begin{figure}[h]
\caption{Dilation technique to capture the mean curvature of codimension $1$ interfaces in $\R^d$, $d=2,3$. }
\vspace{-2mm}
\begin{subfigure}[b]{0.49\textwidth}
    \centering
     \caption{$d=2$} \includegraphics[height=1.7in,width=2.2in]{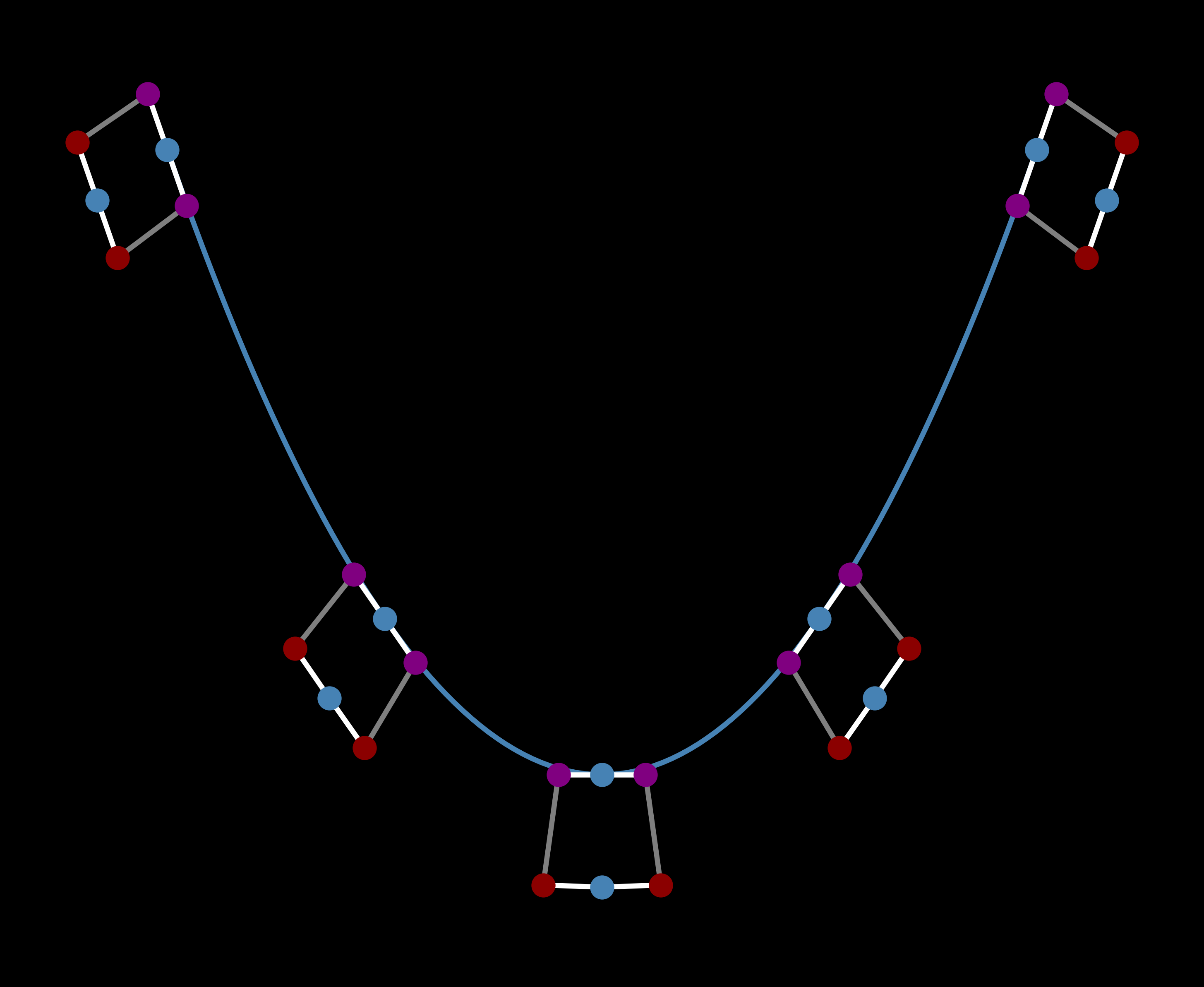}
    \label{fig:curvApprox2D}
\end{subfigure}
\begin{subfigure}[b]{0.49\textwidth}
    \centering
    \caption{$d=3$} \includegraphics[height=1.7in,width=2.2in]{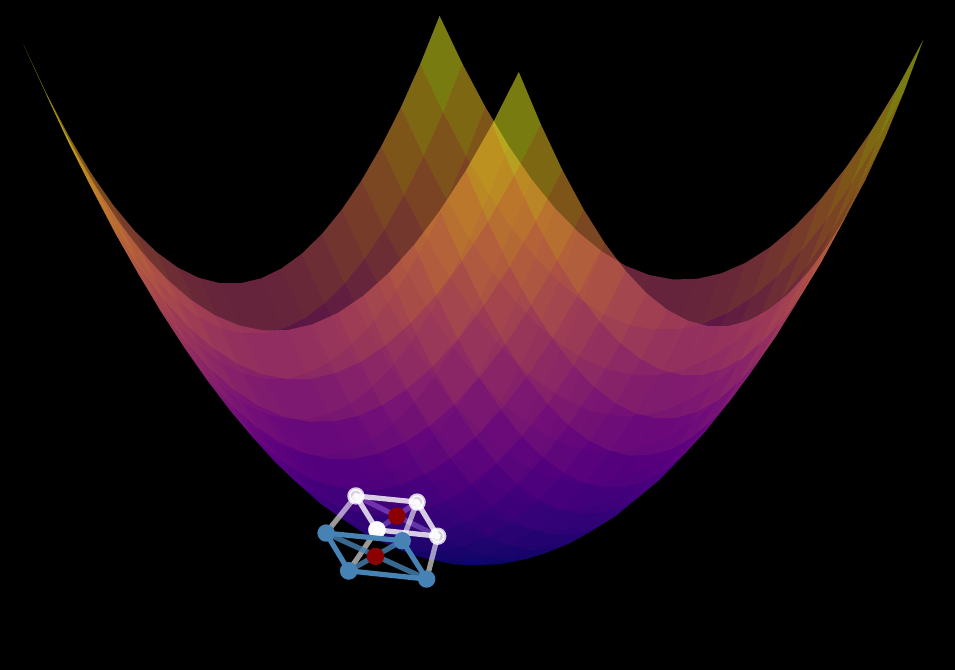}
    \label{fig:curvApprox3D}
\end{subfigure}
\label{fig:curvApprox}
\end{figure}

As a motivation, consider a circle $\Sigma_r = \partial B_r \subseteq \R^2$  of radius $r$ whose curvature is $\kappa_{\Sigma_r} = \frac{1}{r}$. One way to approximate the curvature is to look at the ratio between  arc lengths of $\Sigma_r$   and the dilated circle $\Sigma_{r+\epsilon}$ for small $\epsilon > 0$. Indeed, writing
$\Sigma_r \supset \calA_{r,\epsilon_0} = \{(r\cos(t\epsilon_0),r\sin(t\epsilon_0)) \! : t \in [0,1]\}$ for an arc with radius $r$  and angle $\epsilon_0 \in [0,2\pi)$ we see that
\begin{equation}\label{eq:curvArc}
\frac{\calH^1(\calA_{r+\epsilon,\epsilon_0})}{\calH^1(\calA_{r,\epsilon_0})} = \frac{\epsilon_0(r+\epsilon)}{\epsilon_0 r} = 1 +\epsilon \ \kappa_{\Sigma_r} \; \Longrightarrow\; \kappa_{\Sigma_r} = \frac{1}{\epsilon}\bigg(\frac{\calH^1(\calA_{r+\epsilon,\epsilon_0})}{\calH^1(\calA_{r,\epsilon_0})} -1 \bigg).
\end{equation}

We can mimic this procedure for general subsets $\Sigma  \subset \R^2$, $\text{dim}_{\calH}(\Sigma) = 1$, by approximating $\Sigma$ locally by a circle. The procedure is summarized in \cref{alg:curv2D} where $\Sigma$ is  the zero level set of some function $\phi \in \calC^1(\R^2)$. An illustration is also given in \cref{fig:curvApprox2D}.  When $\epsilon_0 > 0$ is small, the arc of length $\epsilon_0$ in \eqref{eq:curvArc} is accurately replaced by a segment tangent to $\Sigma$. We recall that the outward normal vector field of $\Sigma$ is readily available from  the level-set function. The tangent vector field is therefore available as well. 

\begin{algorithm}[h]
\caption{(Local approximation of  curvature, $d=2$) }\label{alg:curv2D}

\vspace{1mm}
\textbf{Given}:  $\phi\in \calC^1(\Omega)$ $(\nabla \phi \ne 0)$,  $ \ \Sigma = \{\phi = 0 \} $, $ \ y \in \Sigma$, $ \ \epsilon_0>0 , \ \epsilon> 0$\\[-0.8em] 
\noindent\rule{\textwidth}{0.5pt} 
\begin{itemize}
 \setlength \itemsep{-0.1em}
 
\item[I.]  \textbf{Pick} a direction $\mu$ of the tangent line at $y$, i.e., perpendicular to $\nu(y) = \frac{\nabla \phi(y)}{|\nabla \phi(y)|}$ 
\item[II.] \textbf{Consider} \\[-2em]
\begin{enumerate}
    \item the segment $S_y \subseteq \R^2$ with endpoints $y^{\pm}= y \pm \epsilon_0 \mu$ 
    \item the dilated segment $S^{\epsilon}_y \subseteq \R^2$ with endpoints $y^{\pm} + \epsilon \  \nu(y^{\pm} )$
\end{enumerate}

\item[III.] \textbf{Compute} the ratio $\frac{\mathcal{H}^1(S^{\epsilon}_y)}{\mathcal{H}^1(S_y)} \approx 1 + \epsilon \ \kappa_{\Sigma}(y)$

\item[IV.] \textbf{Return} the curvature $\kappa_{\Sigma}(y) \approx \frac{1}{\epsilon}\Big(\frac{\mathcal{H}^1(S^{\epsilon}_y)}{\mathcal{H}^1(S_y)} - 1 \Big)$\\[-1em]
\end{itemize}

\end{algorithm}

A similar procedure can be used in the three-dimensional case, as explained in \cref{alg:curv3D} and displayed in \cref{fig:curvApprox3D}. In short, the segments in \cref{alg:curv2D} become quadrangles   and lengths are replaced by areas. 
In Step III of \cref{alg:curv3D}, note that the mean curvature  appears naturally from the cross products in $(1+\epsilon \kappa_{\Sigma,1})(1+\epsilon \kappa_{\Sigma,2})$, where $(\kappa_{\Sigma,i})_{i=1}^2$ are the principal curvatures of $\Sigma$. 
As the principal directions are unknown in general, the algorithm picks randomly two orthogonal vectors in the plane tangent to $\Sigma$ at some point $y\in \Sigma$. This randomization has little impact on the accuracy of the obtained, as can be seen in \cref{fig:curvParabol2D} and \cref{fig:curvParabol3D}. 

\begin{algorithm}[h]
\caption{(Local approximation of mean curvature, $d=3$) }\label{alg:curv3D} 
\vspace{1mm}
\textbf{Given}: $\phi\in \calC^1(\Omega)$ $(\nabla \phi \ne 0)$,  $ \ \Sigma = \{\phi = 0 \} $, $ \ y \in \Sigma$, $ \ \epsilon_0>0 , \ \epsilon> 0$\\[-0.8em] 
\noindent\rule{\textwidth}{0.5pt} 
\begin{itemize}
 \setlength \itemsep{-0.1em}
 
\item[I.] \textbf{Pick} two directions  $\mu_1 \perp \mu_2$ in the tangent plane at $y$, i.e., perpendicular to  $ \nu(y) = \frac{\nabla \phi(y)}{|\nabla \phi(y)|}$
\item[II.] \textbf{Consider} \\[-2em]
\begin{enumerate}
\item the quadrangle $Q_y$ with vertices  $y_i^{\pm} = y \pm \epsilon_0 \mu_i, \ i =1,2$ 

    \item  the quadrangle $Q^{\epsilon}_y$ with vertices $y_i^{\pm} + \epsilon \  \nu(y_i^{\pm} ),\ i =1,2$ 
\end{enumerate}

\item[III.] \textbf{Compute} the ratio $\frac{\mathcal{H}^2(Q^{\epsilon}_y)}{\mathcal{H}^2(Q_y)} \approx (1 + \epsilon \ \kappa_{\Sigma,1}(y))(1 + \epsilon \ \kappa_{\Sigma,2}(y)) \approx 1 + 2 \epsilon \ \kappa_{\Sigma}(y)$

\item[IV.] \textbf{Return} the  mean curvature $\kappa_{\Sigma}(y) \approx \frac{1}{2\epsilon}\Big(\frac{\mathcal{H}^2(Q^{\epsilon}_y)}{\mathcal{H}^2(Q_y)} - 1 \Big)$ \\[-1em]
\end{itemize}
\end{algorithm}

    Back to the Stefan problem, we  naturally choose $\Sigma = \partial \Gamma_t$, $t\in [0,T]$, and $\epsilon \ll \epsilon_0 \ll \delta$, where $\delta$ is the width of the fuzzy regions $\partial \Gamma_t^{\delta,1}$, $\partial \Gamma_t^{\delta,2}$. Indeed, $\epsilon \ll \epsilon_0$ is imposed to gain accuracy while $\epsilon_0 \ll \delta$ ensures that the points $y^{\pm}$ in \cref{alg:curv2D} ($y_i^{\pm}$ in  \cref{alg:curv3D}) belong to the fuzzy region for all $y\in \partial \Gamma_t$. 

Let us verify \cref{alg:curv2D,alg:curv3D}, respectively, 
for parabolas $\Sigma_a := \{y  \in \R^2 \! : y_2 = \frac{y_1^2}{a} \}$ and paraboloids $\Sigma_{a,b} := \{y  \in \R^3 \! : y_3 = \frac{y_1^2}{a} + \frac{y_2^2}{b} \}$ with parameters $a,b \in \R\setminus\{0\}$. The mean curvature functions are given, respectively, by
\begin{equation}\label{eq:curvParabolExact}
    \kappa_{\Sigma_a}(y) 
= 
    \frac{2}{a\left(1+ \frac{4y_1^2}{a^2}\right)^{3/2}}\quad  (d=2), \qquad \kappa_{\Sigma_{a,b}}(y) = \frac{\frac{4y_1^2}{a} +  \frac{4y_2^2}{b} + a + b}{ab\left(1+ \frac{4y_1^2}{a^2} +  \frac{4y_2^2}{b^2}\right)^{3/2}} \quad  (d=3).
\end{equation}
\cref{fig:curvParabol2D} shows the approximated curvature and the error relative to \eqref{eq:curvParabolExact} for a parabola $(d=2)$ with parameter $a=2$. In \cref{fig:curvParabol3D}, we repeat the exercise  for a  paraboloid with parameters $(a,b)=(1,2)$ (top panels) and a hyperbolic paraboloid with  $(a,b) =(-1,2)$ (bottom panels). As expected, we note that the error for the hyperbolic paraboloid in \cref{fig::errorParaboloid2} is only sizeable along the points with zero mean curvature. 
\begin{figure}[h]
\caption{Estimated  curvature (left panel) and relative error (right panel) for a parabola $(d=2)$ with parameter $a = 2$. }
\vspace{-2mm}

\begin{subfigure}[b]{0.49\textwidth}
    \centering
     \caption{Curvature} \includegraphics[height=1.7in,width=2.2in]{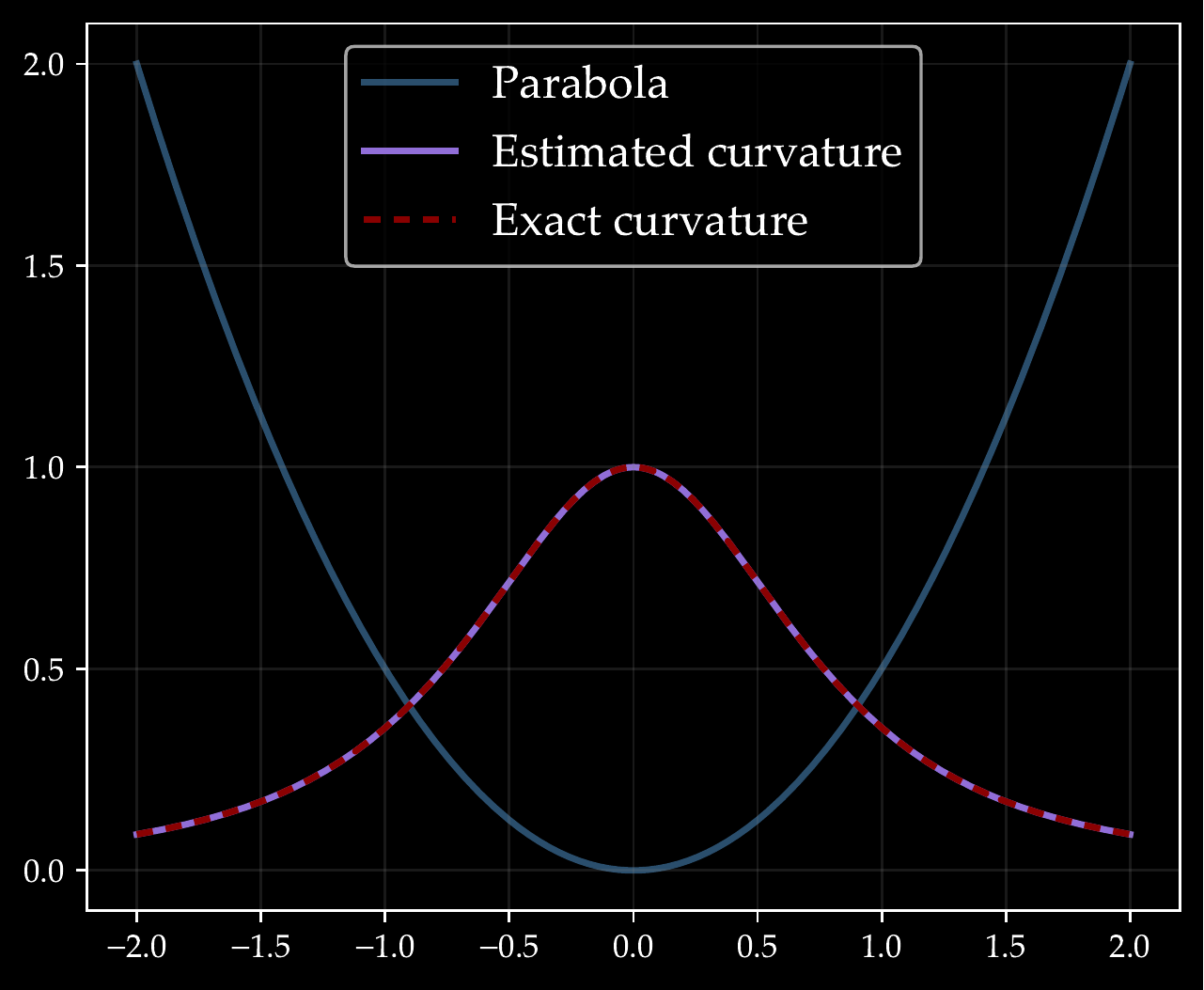}
    \label{fig:curvParaboloid1}
\end{subfigure}
\begin{subfigure}[b]{0.49\textwidth}
    \centering
    \caption{Relative error} \includegraphics[height=1.7in,width=2.2in]{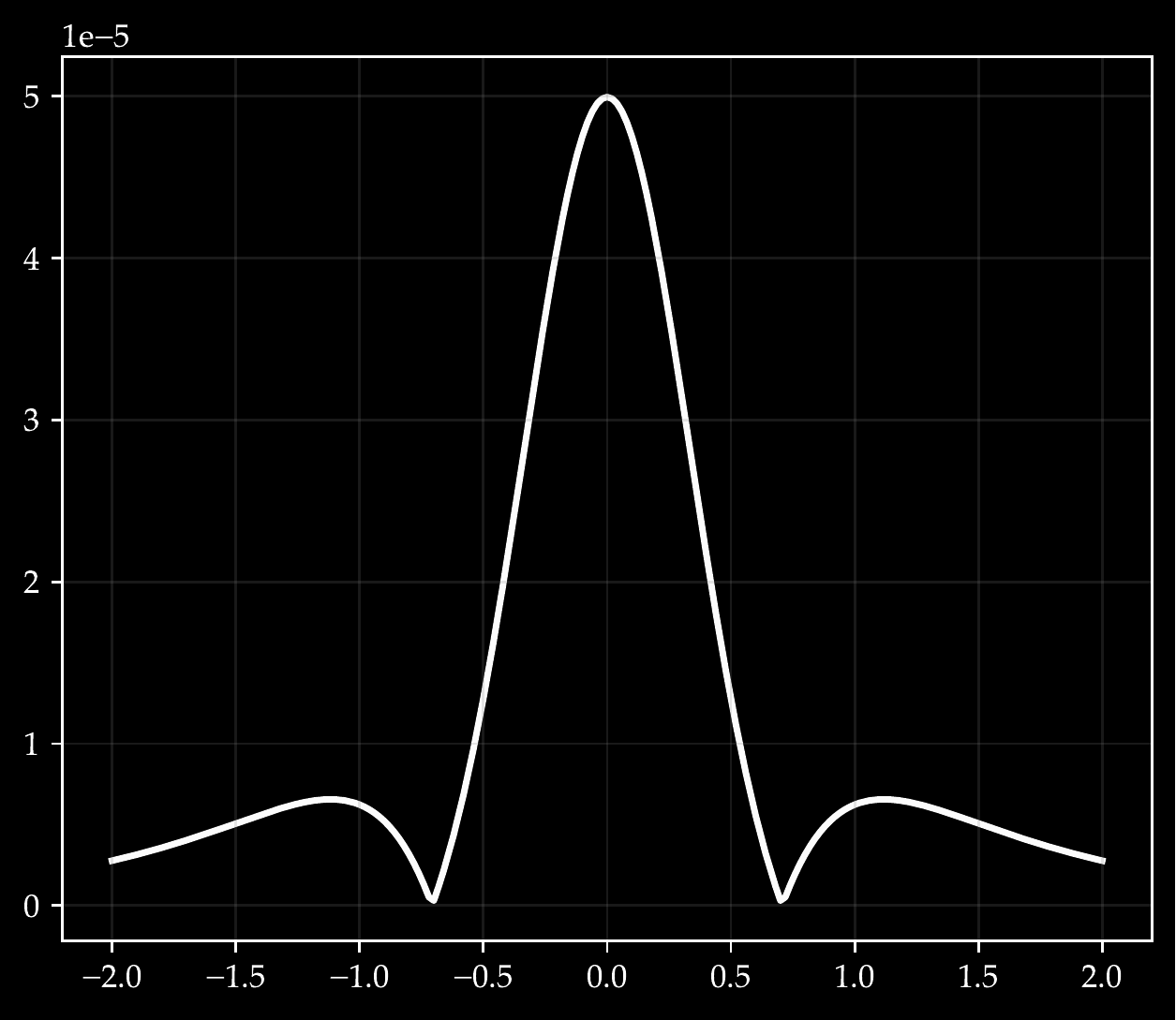}
    \label{fig::errorParaboloid1}
\end{subfigure}

\label{fig:curvParabol2D}
\end{figure}

\begin{figure}[H]
\caption{Estimated mean curvature (left panels) and relative error (right panels) for a paraboloid $(d=3)$ with parameters $(a,b) = (1,2)$ (top panels) and $(a,b) = (-1,2)$ (bottom panels). }
\vspace{-2mm}

\begin{subfigure}[b]{0.49\textwidth}
    \centering
     \caption{Mean Curvature, $(a,b) = (1,2)$} \includegraphics[height=1.7in,width=2.2in]{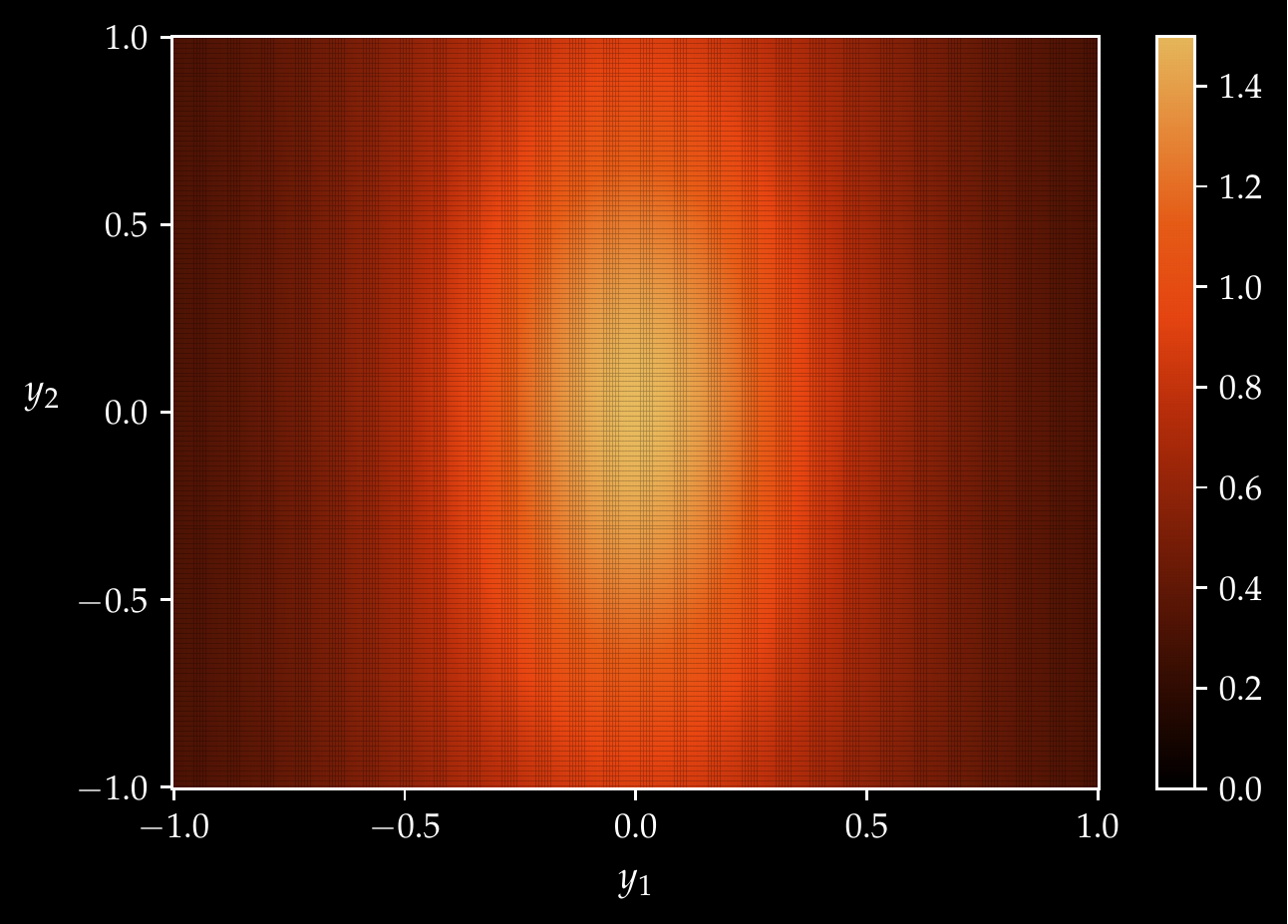}
    \label{fig:curvParaboloid1}
\end{subfigure}
\begin{subfigure}[b]{0.49\textwidth}
    \centering
    \caption{Relative error, $(a,b) = (1,2)$} \includegraphics[height=1.7in,width=2.2in]{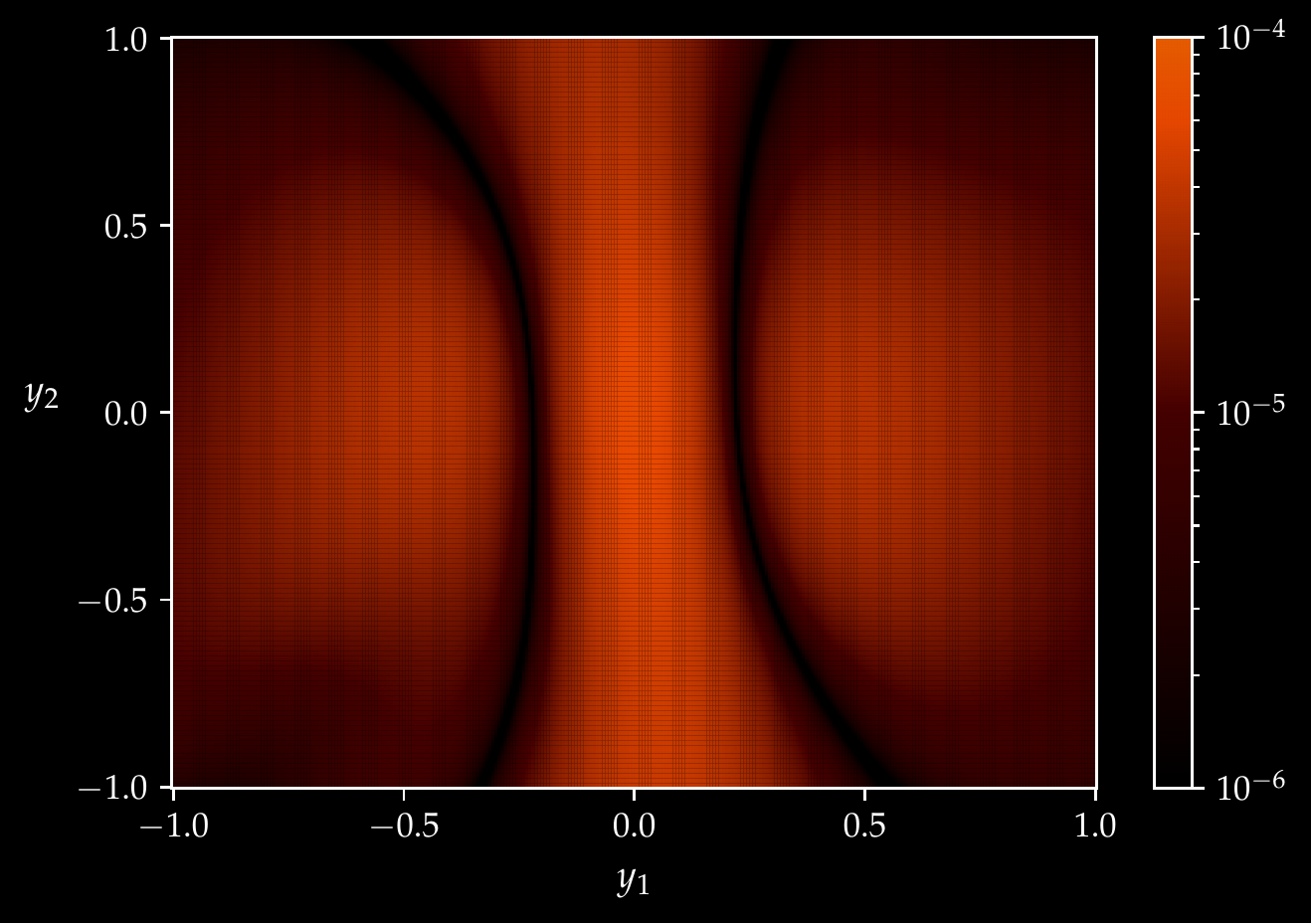}
    \label{fig::errorParaboloid1}
\end{subfigure}
\vspace{3mm}

\begin{subfigure}[b]{0.49\textwidth}
    \centering
     \caption{Mean Curvature, $(a,b) = (-1,2)$} \includegraphics[height=1.7in,width=2.2in]{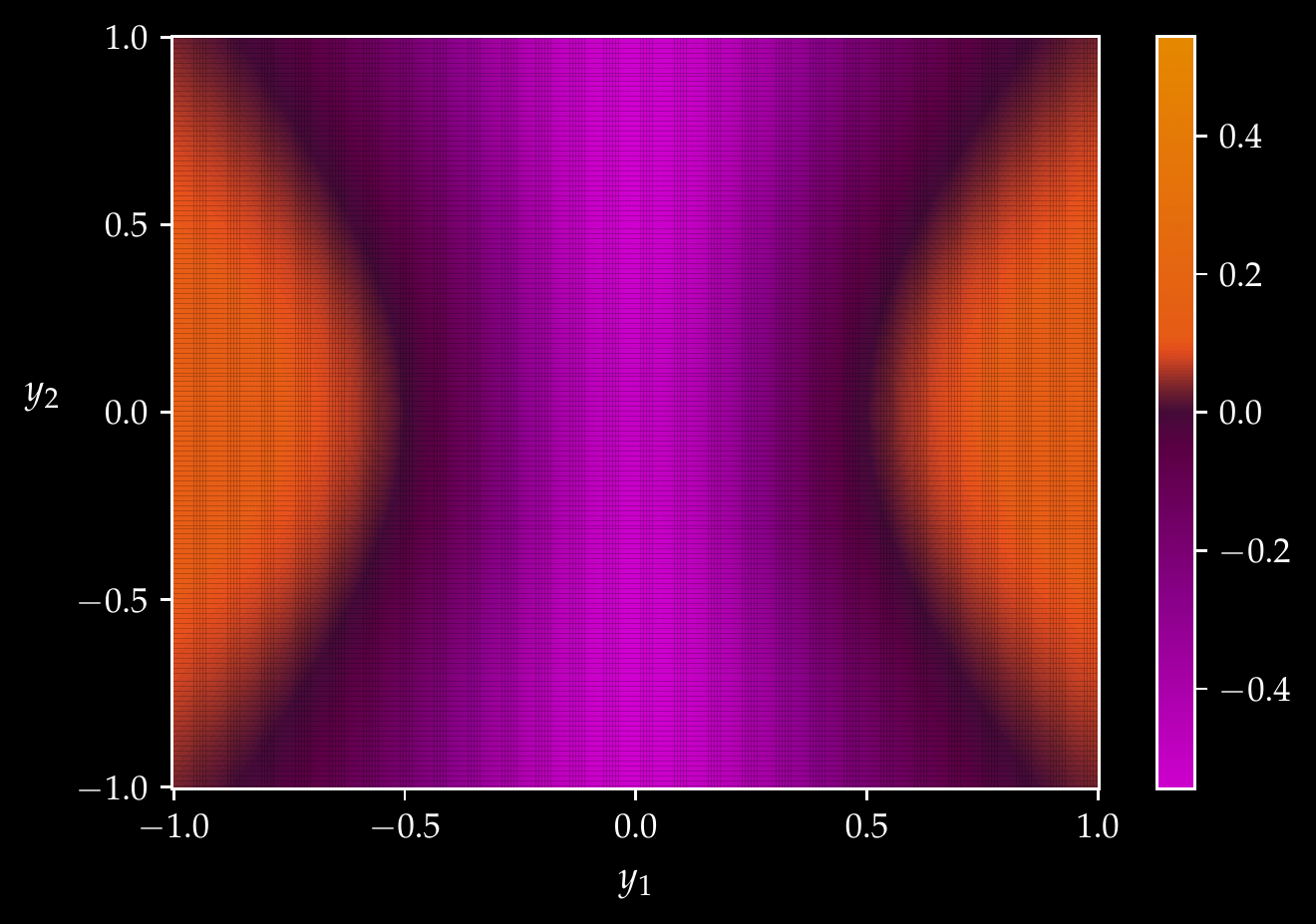}
    \label{fig:curvParaboloid2}
\end{subfigure}
\begin{subfigure}[b]{0.49\textwidth}
    \centering
    \caption{Relative error, $(a,b) = (-1,2)$} \includegraphics[height=1.7in,width=2.2in]{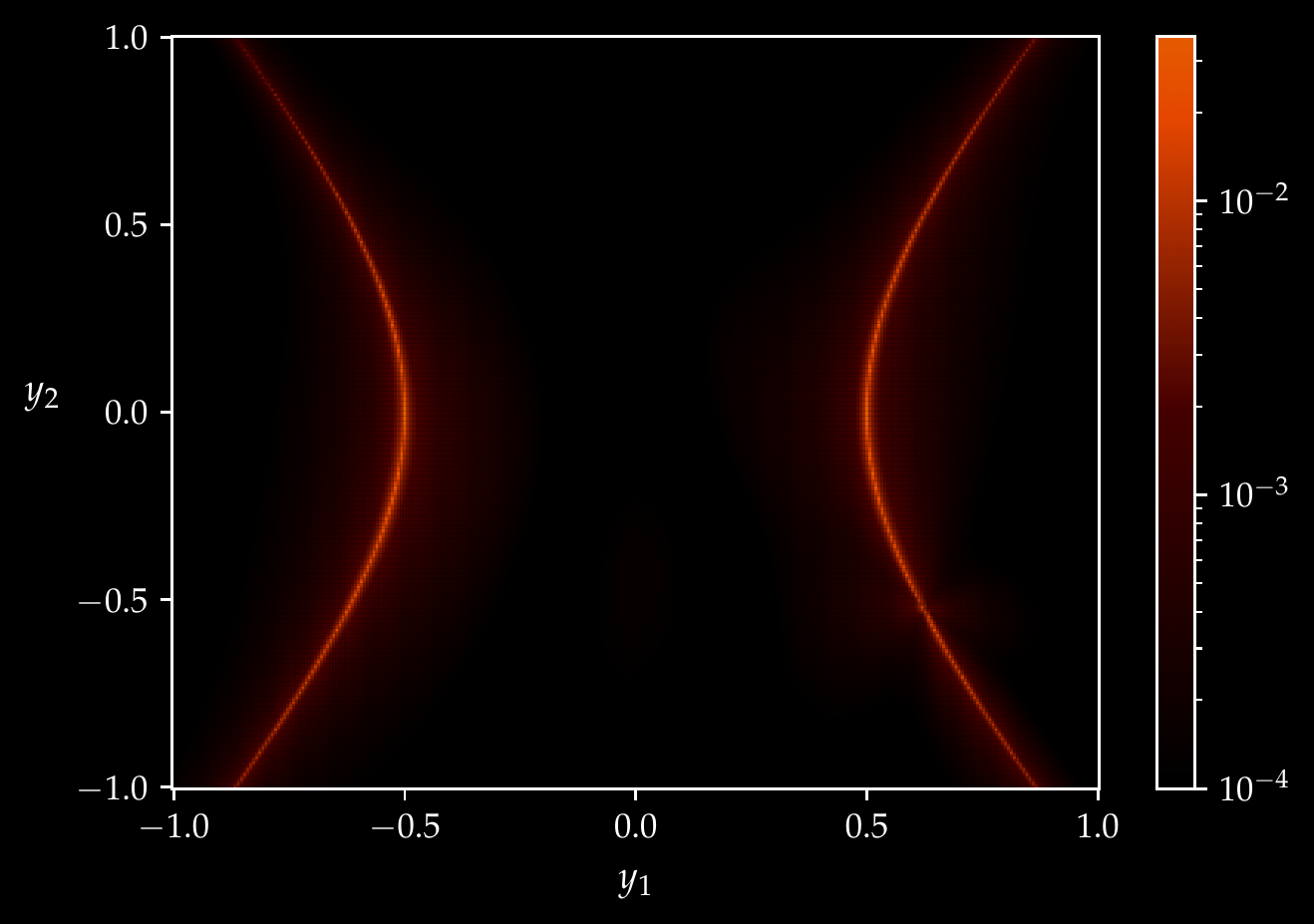}
    \label{fig::errorParaboloid2}
\end{subfigure}

\label{fig:curvParabol3D}
\end{figure}



\section{Numerical Results}\label{sec:Numerics}

\cref{tab:trainingParams} and \cref{tab:problemParams} report the default training and problem parameters, respectively, across the examples. The value of $M$ in \cref{tab:trainingParams} ($3000$) indicates the \textit{maximum} number of iterations if the loss function has not stabilized yet.  For the algorithm with surface tension, we set $M = 1000$ as each training iteration is more costly. Note that the Lagrange multiplier in front of the penalty function is  less than one to give more importance to the loss term than the penalty. The deep-level set function in 
\eqref{eq:deepLevelSet} consists of  a feedforward neural network with input $(t,x)$ of size $\bar{d} = d+1$ and two hidden layers with $20+\bar{d} \ $ hidden nodes. The code is implemented in 
Python 3.9 using Tensorflow 2.7 and run on CPU ($10$ cores)   
  on a 2021 Macbook pro with 64GB unified memory 
and Apple M1 Max chip. 
\begin{table}[H]
\begin{center}
\caption{Default Training Parameters ($d =$ number of space dimensions)}

\begin{tabular}{lll}
\hline \hline
Parameter & Definition & Value \\ \hline \hline
 $J$  & Batch size & $2^{7+d}$ \\ 
$M$ & Number of training iterations  & $3000$ \\ 
 $K$ & Number of test functions $(|\Psi|)$    & $100 d$ \\
 $N$ & Number of time steps  & $100$ \\ 
 $\varepsilon_i$ & Mushy region width of phase $i\in \{1,2\}$ &  $\sqrt{\alpha_i d   \frac{T}{N}}$ 
 \\
 $\lambda_0$ & Lagrange multiplier & $0.1$\\ 
 \hline\\[-1em]

\end{tabular}

\label{tab:trainingParams}
\end{center}
\end{table}

\begin{table}[H]
\begin{center}

\caption{Default problem parameters}
\begin{tabular}{lll}
\hline \hline
Parameter & Definition & Value \\ \hline \hline
 $T$  & Time horizon & $1$ \\ 
 $R$ & Radius of spherical domain $\Omega$    & $1$ 
 \\
  $\alpha_1$ & Diffusivity of liquid particles    & $0.5$ \\
   $\alpha_2$ & Diffusivity of solid particles    & $0.5$ \\
  $\gamma$ & Surface tension coefficient & $0$ \\
\hline \\[-1em]
\end{tabular}
\vspace{5pt}

\label{tab:problemParams}
\end{center}
\end{table}




\subsection{Radial Case}\label{sec:RadialNumerics}
We first discuss examples where the solid is radially symmetric. For concreteness, we regard the solid as an ice ball surrounded by (liquid) water.  

\subsubsection{One-phase, Melting Regime  \textnormal{$(d=2)$}}
\label{sec:Hadzic}
Consider the  one-phase  Stefan problem without surface tension as in \cref{ex:OnePhase}.   Given $\Gamma_{0-} = B_{1/2} \subset B_{1} = \Omega$, and assuming the initial temperature  $u_0^1$ to be radially symmetric, the solid (ice) remains radial as well, i.e., $\Gamma_t = B_{r(t)}$ for some càdlàg function $r\!:[0,T] \to [0,R]$.  
 We also set $\eta = 1$, and  $L = 1/4$. The temperature in the liquid (water) is initially constant, namely $u^1_0 \equiv \frac{1}{|B_R\setminus B_{r_0}|}$. 
 
\cref{fig:1Phasemahir} compares the  behavior of $r(t)$ as $t$ approaches the melting time with the  theoretical asymptotic rate given by \citet{Hadzic}, namely 
\begin{equation}\label{eq:meltRate2D1P}
    r(t) \  \sim \  \sqrt{\tau_{\varnothing}-t}\, e^{- \sqrt{ \frac{1}{2} |\log (\tau_{\varnothing}-t)|}}, \quad t \ \uparrow \ \tau_{\varnothing} := \inf\{t\ge 0:\,\Gamma_t = \varnothing\}. 
\end{equation}
As can be observed, the melting rate obtained with the deep level-set method is indeed close to the theoretical one. 




\begin{figure}[h]
    \centering
    \caption{One-phase Stefan problem in the radially symmetric case ($d=2$). Evolution of the solid and the liquid regions over time. }
    \includegraphics[height = 1.4in,width = 5.8in]{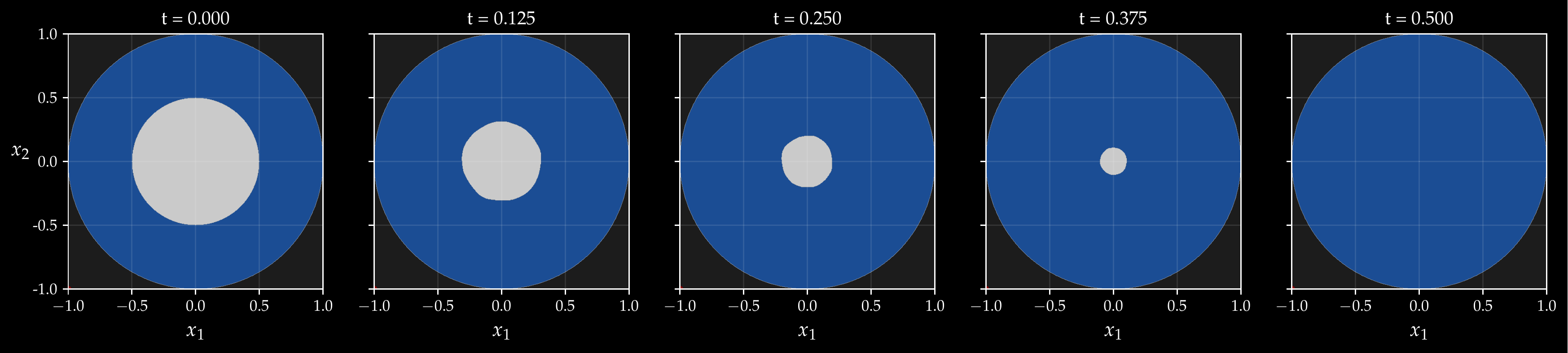}
    \label{fig:1PhaseMelting}
\end{figure}

\begin{figure}[h]
\caption{Mean radius of the solid region over time (left panel) and zoomed in on the melting time  (right panel). The light purple curves represent one standard deviation above and below the mean radius for various angles. The orange dashed curve is the theoretical  melting rate from  \cite{Hadzic}.} 
\vspace{-2mm}
\begin{subfigure}[b]{0.49\textwidth}
    \centering
 \includegraphics[height=1.8in,width=2.6in]{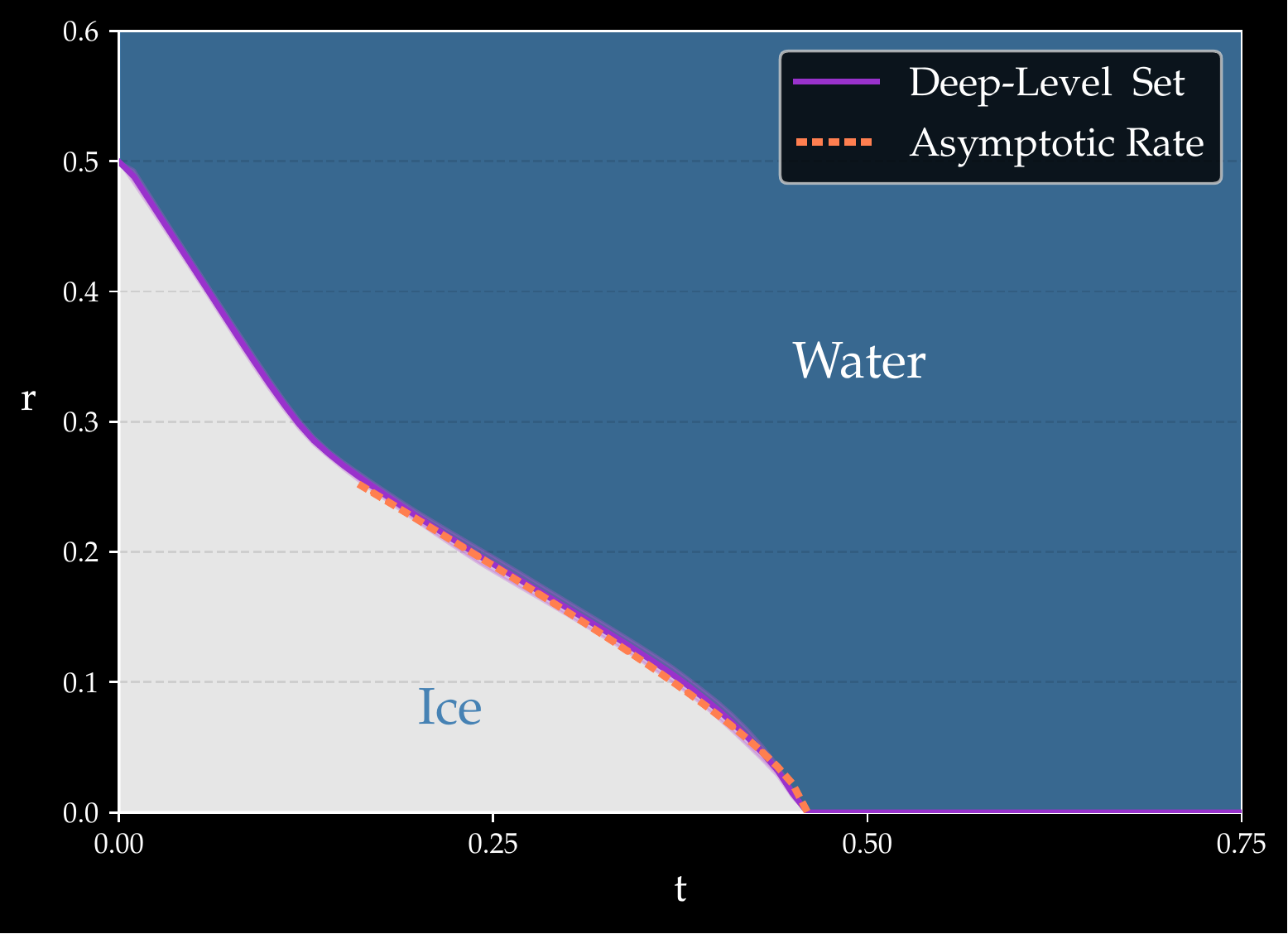}
    \label{fig:1Phasemahir}
\end{subfigure}
\begin{subfigure}[b]{0.49\textwidth}
    \centering
\includegraphics[height=1.8in,width=2.4in]{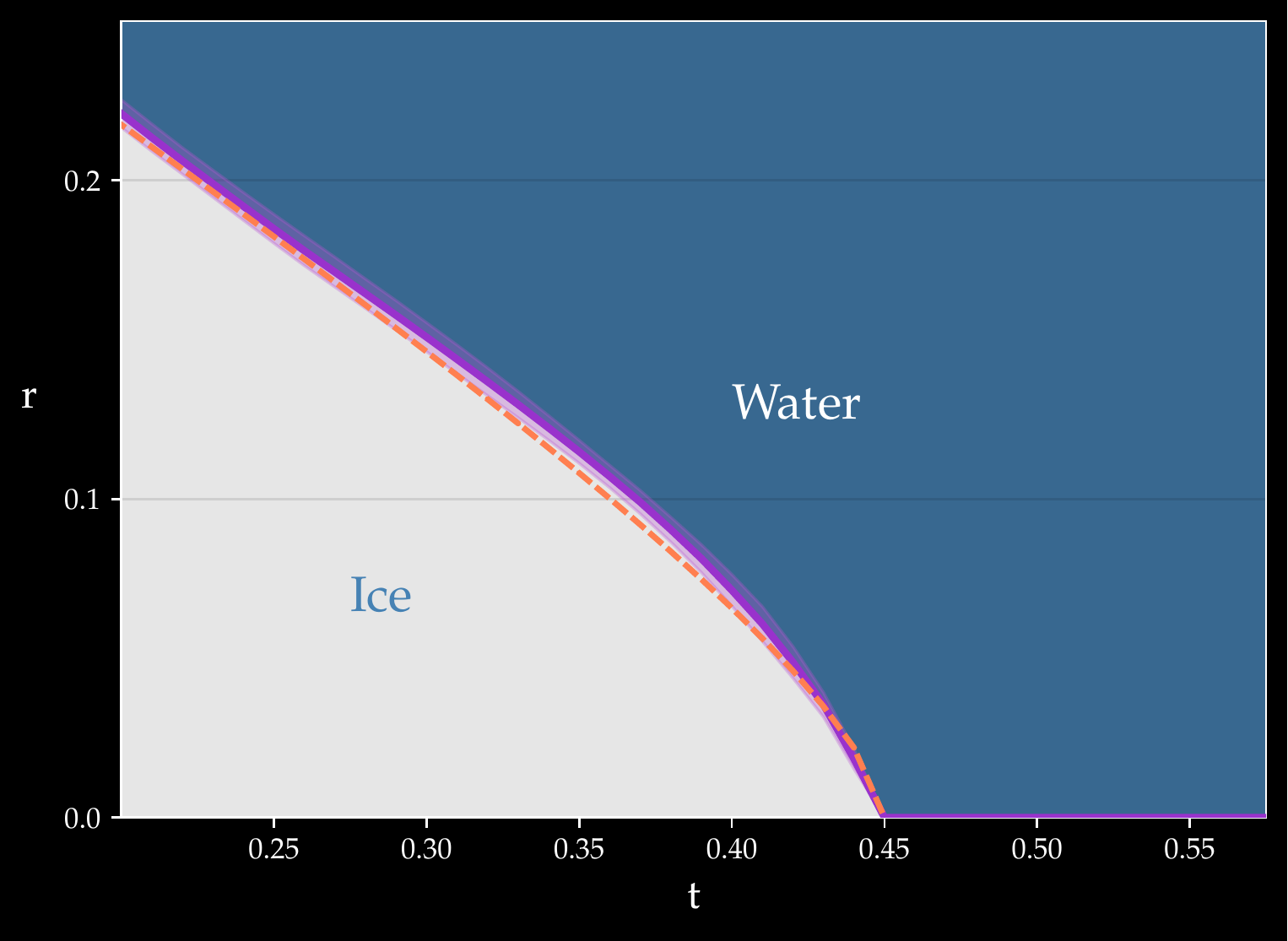}
    \label{fig:1PhasemahirZoom}
\end{subfigure}
\end{figure}

\subsubsection{Two-phase, Freezing Regime \textnormal{$(d=2)$}} \label{sec:longterm}
Let $d=2$ and consider the two-phase, radially symmetric supercooled Stefan problem without surface tension. Suppose that the temperature is initially constant in the liquid and the solid regions, namely 
$$u^1_0 = -\frac{\mathds{1}_{B_R\setminus B_{r_0}}}{2|B_R\setminus B_{r_0}|}, \qquad u^2_0 = -\frac{1}{10|B_{r_0}|}\mathds{1}_{B_{r_0}}. $$
We therefore choose the constants $c_1 = 1/2, \ c_2 = 1/10$ in \cref{rem:initConstants}. In this example, the radius of the solid region converges to some constant $r_{\infty} \in (0,R)$. In other words, the solid region neither melts completely nor covers the whole domain in the long run. In light of the Dirichlet condition \eqref{eq:2SPDirichlet}, the long-term temperature must be constant and equal to $u^1(\infty,\cdot)\equiv u^2(\infty,\cdot) \equiv 0$. In fact, the theoretical value for $r_{\infty}$ can be derived  from the Stefan growth condition. Indeed, choosing an increasing sequence $\psi_{m} \in \calC^{\infty}_c(\Omega)$ such that $\psi_{m} \uparrow 1$ in $\Omega$ and time integrating \eqref{eq:intermediateGrowthCond} in the proof of \cref{prop:GrowthCondNoTension} between $t=0-$ and $t=T\to \infty$ yields 
\begin{align}  
|\Gamma_{\infty}| -  |\Gamma_{0-}|  =  \frac{1}{L}\sum_{i=1}^2\bigg(\int_{\Gamma_{\infty}^i} u_{\infty}^i - \int_{\Gamma_{0-}^i} u_{0}^i \bigg) \; \Longleftrightarrow \; \pi r_{\infty}^2 - \pi r_0^2  = \frac{c_1 + c_2}{L}. 
\end{align}
The long-term radius is therefore $r_{\infty} = \sqrt{r_0^2 + \frac{c_1 + c_2}{L \pi}}$. 
We can thus solve for $r_{\infty}$ and compare it with the obtained long-term radius. 
We choose $r_0 = 1/2$,  $L = 4$, and $T=5$. The long-term radius is roughly equal to  $0.59$, as confirmed in \cref{fig:2D2PLongTerm}.  

\begin{figure}[t]
    \centering
    \caption{Two-phase, supercooled Stefan problem without surface tension ($d=2$) discussed in \cref{sec:longterm}.  The radius indeed converges to the theoretical long-term value $r_{\infty} \approx 0.59$. }
    \includegraphics[height = 2in,width = 3in]{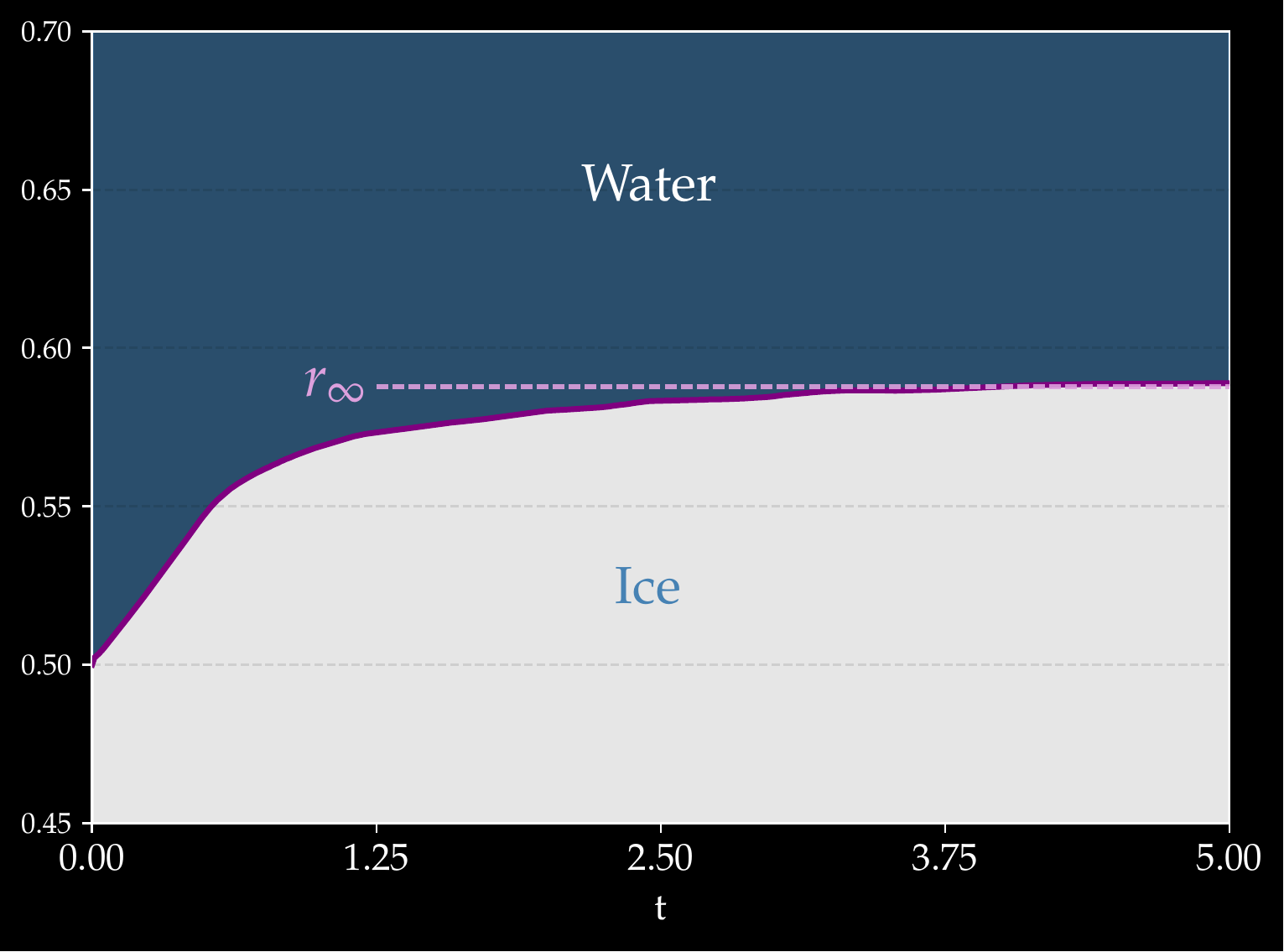}
    \label{fig:2D2PLongTerm}
\end{figure}

Let us now estimate the temperature in both phases as a function of time and radius from the trained moving solid. In view of \eqref{eq:UvsProb}, we here have  that  $u^i(t,x)\,\mathrm{d}x = -c_i\Pb(X_t^i \in \mathrm{d}x, \tau^i>t)$ for $t\in [0-,T]$, $i=1,2$. Writing $v^i(t,|x|) :=  u^i(t,x) $ and using polar coordinates yields 
\begin{align*}
    v^i(t,r) &=  - \frac{c_i}{2 \pi r} \underbrace{\frac{\Pb(|X_t^i| \in \mathrm{d}r \ | \ \tau^i > t)}{\mathrm{d}r}}_{(a)} \ \underbrace{\Pb( \tau^i > t)}_{(b)}.   
\end{align*}
The term $(a)$ is the time $t$ conditional density of the norm of surviving particles which can be approximated using Monte Carlo simulation and kernel density estimation (KDE). The second term $(b)$ is  the unconditional survival probability over time and can also be estimated using simulation.
\cref{fig:2D2PLongTermTemp} displays the temperature of the solid and the liquid over time. As can be seen, the temperature at time $t=2$ is already close to its  equilibrium $u^1_{\infty} = u^2_{\infty} = 0$. 

\begin{figure}[t]
    \centering
    \caption{Two-phase, supercooled Stefan problem without surface tension ($d=2$) discussed in \cref{sec:longterm}. Temperature (vertical axis) in the liquid ($u^1$) and the solid ($u^2$)  retrieved using Monte Carlo simulation and kernel density estimation (KDE). At $t=3$, both phases are already close to their equilibrium temperature $u^1_{\infty} = u^2_{\infty} = 0$  and long term term radius $r_{\infty} \approx 0.59$.  }
    \includegraphics[height = 1.4in,width = 5.8in]{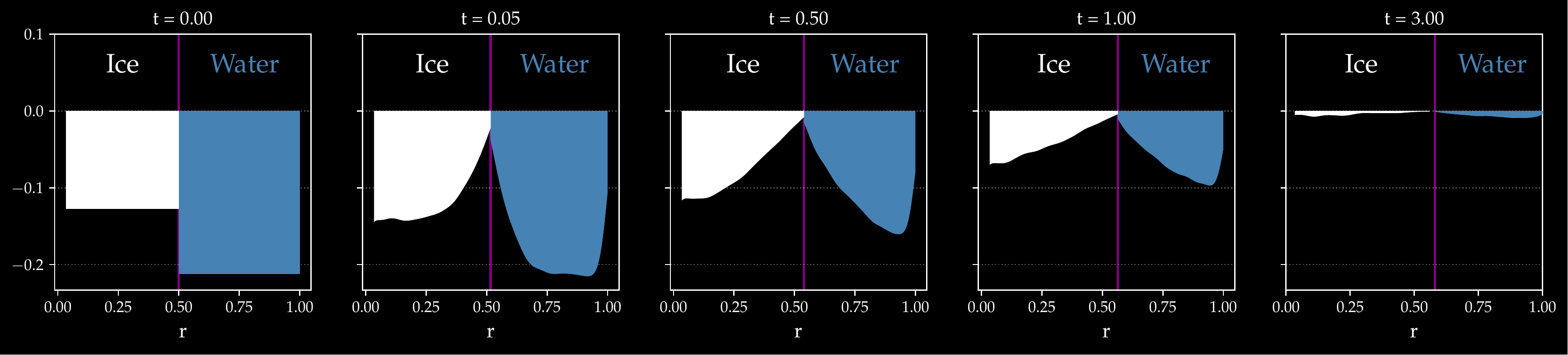}
    \label{fig:2D2PLongTermTemp}
\end{figure}

\subsubsection{Two-phase, with Surface Tension \textnormal{$(d=3)$}}\label{sec:3D2PTension}
Consider the three-dimensional Stefan problem with surface tension. The initial temperatures in the solid and the liquid are respectively given by 
$$u^1_0 = -\frac{\mathds{1}_{B_R\setminus B_{r_0}}}{ 2|B_R\setminus B_{r_0}|}, \quad u^2_0 = - \frac{\mathds{1}_{B_{r_0}}}{|B_{r_0}|}. $$
We therefore have $\eta = -1$ (supercooled liquid),  $c_1 = 1/2$, and $c_2 = 1$.  
Again, the solution remains radially symmetric, i.e., $\Gamma_t = B_{r(t)}$ for some function $r\!:[0,T] \to [0,R]$. The temperature at the interface is therefore $u^1\equiv u^2 \equiv -\frac{\gamma}{r(t)}$. We  use the growth condition in \cref{prop:growthCondTensionPoisson} to train the parameters and estimate the mean curvature via \cref{alg:curv3D}. To benchmark our method, we apply a useful trick in the three-dimensional radial case to get rid of surface tension. A similar argument is given in \cite[Section 4]{HHV}. For $i=1,2,$ define 
\begin{equation}\label{eq:trick}
    \widetilde{u}^i(t,x) = u^i(t,x)+ \frac{\gamma}{|x|}.
\end{equation}
Then $\widetilde{u}^1 \equiv \widetilde{u}^2 \equiv 0$ on $\partial \Gamma$, so 
the Gibbs-Thomson condition \eqref{eq:GibbsThomson} is effectively ``absorbed'' by the transformation in \eqref{eq:trick}. In addition,  $\Delta\left(|x|^{-1}\right)= (3-d)|x|^{-3}$, hence  $x \mapsto |x|^{-1}$ is harmonic if $d=3$ and 
$\widetilde{u}^i$ still solves the heat equation \eqref{eq:2SPHeat}. We can therefore apply the deep level-set  method without surface tension  by changing the initial condition to $\widetilde{u}^i_0(x) =u^i_0(x) + \frac{\gamma}{|x|} $. It is worth noting that the Stefan condition \eqref{eq:2SPStefan} is also affected. As $u^i = \widetilde{u}^i - \frac{\gamma}{|x|} $ and  $\partial_{\nu}(\frac{1}{|x|}) = -\frac{1}{|x|^2}$, we obtain
$$ V =  \frac{\alpha_2}{2L}\partial_\nu u^2 -\frac{\alpha_1}{2L} \partial_\nu u^1 =   \frac{\alpha_2}{2L}\partial_\nu \widetilde{u}^2 -\frac{\alpha_1}{2L} \partial_\nu \widetilde{u}^1  + \gamma \frac{\alpha_2 - \alpha_1}{2L|x|^2}.$$
If  $\alpha_1 = \alpha_2$, then the Stefan condition remains unchanged. Otherwise, when $t\mapsto r(t)$ is strictly monotone, the growth condition in \eqref{eq:probaCond2} becomes 
\begin{align*}
& \int_{\Gamma_{0-}}\psi -  \int_{\Gamma_t  } \psi 
   =  \frac{1}{L} \Big(\E^{\mu^1}[  \psi(X_{\tau^1}^1)\,\mathds{1}_{\{ \tau^1 \le t\}}] - \E^{\mu^2}[  \psi(X_{\tau^2}^2)\,\mathds{1}_{\{ \tau^2 \le t\}}] +  \calK_t \Big) ,\quad \psi \in \mathcal{C}_c^{\infty}(\Omega), \\[0.5em]
&  \calK_t = \gamma \frac{\alpha_1-\alpha_2}{2} \int_0^t\int_{\partial \Gamma_s} \frac{\psi}{|x|^2} = \gamma \frac{\alpha_1-\alpha_2}{2} \bigg(\int_{ \Gamma_t} \frac{\psi}{|x|^2} - \int_{ \Gamma_0} \frac{\psi}{|x|^2} \bigg).
\end{align*}
In this example, we assume that $\alpha_1=\alpha_2 = \frac{1}{2}$, $L=2$, $\gamma = 0.25$. The above remedy when $\alpha_1 \neq \alpha_2$ is therefore not needed.  We note that after the radial trick the liquid is still supercooled, i.e.,~$\widetilde{u}_0^1\le 0$. 

\cref{fig:3DRadialTension} compares the  radius of the solid over time obtained from the growth condition in \cref{prop:growthCondTensionPoisson} (orange curve), using the radial trick (dashed purple curve), and without surface tension (red curve). The algorithm with surface tension and its benchmark indeed give similar results. As expected, the growth of the solid ball is less pronounced with surface tension as the freezing point becomes negative for convex solid regions. 

\begin{figure}[H]
\caption{Two-phase, radially symmetric Stefan problem ($d=3$) with surface tension $(\gamma = 0.25)$ and without. 
Time evolution of the radius of the solid  obtained from the growth condition in \cref{prop:growthCondTensionPoisson} (orange curve), using the radial trick (dashed purple curve), and without surface tension (red curve).}  
\vspace{-2mm}
\begin{center}
    \includegraphics[height = 2in, width = 2.8in]{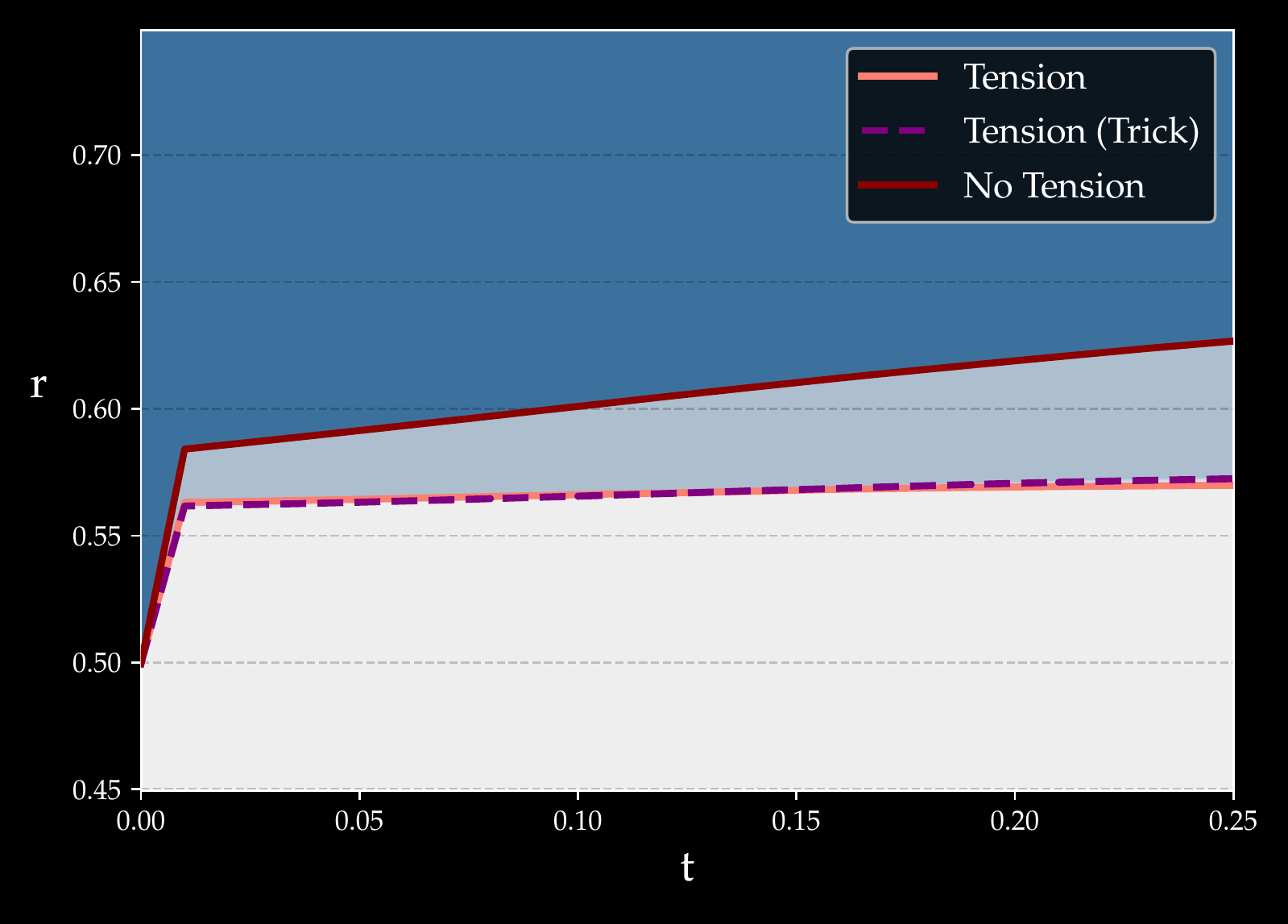}
\end{center}
\label{fig:3DRadialTension}
\end{figure}

\subsubsection{Two-phase,  Jump in the Radius \textnormal{$(d=2)$} } \label{sec:jumpExample}
Consider the two-dimensional radial supercooled Stefan problem inside the ball $\Omega = B_1 \subseteq \R^2$. We set the latent heat to $L=2$ and assume that there is no surface tension at the interface. Given $\Gamma_{0-} = B_{r_0}$, $r_0\in (0,R)$, and assuming the initial temperature to be radial, the solid region remains radial as well, i.e., $\Gamma_t = B_{r(t)}$ for some càdlàg function $[0-,T] \ni t\mapsto r(t)$. The goal of this experiment is to demonstrate the method's ability to generate jumps in the radius $r(\cdot)$. To this end,  consider the following initial condition:
\begin{equation}\label{eq:jumpExInitialTemp}
    u^1_0(x) = -\frac{1}{|A_{r_0,\delta_0 } |}\mathds{1}_{A_{r_0,\delta_0 } }(x), \quad  u^2_0(x) = - \frac{1}{|B_{r_0} |}\mathds{1}_{B_{r_0}}(x), \quad A_{r_0,\delta_0} = B_{r_0+\delta_0}\backslash B_{r_0}, 
\end{equation} 
where $r_0 = \frac{1}{4},\ \delta_0 = \frac{1}{8}$. 
In other words, the liquid is strongly  supercooled inside the annulus $A_{r_0,\delta_0 }$ and at zero temperature elsewhere. 
We may therefore expect  a sizeable liquid region surrounding the solid to freeze immediately, leading to an initial positive jump in the radius. This is indeed the case as observed in \cref{fig:2D2PJump}.   
As seen in \cref{sec:jumpPenalty}, we can in fact quantify the magnitude of the jump for physical solutions.  In this case, we have from  \eqref{eq:eqPosJump} that  $\Delta r(0)$ must be the smallest positive value (if any) such that 
\begin{align*}
  |A_{r_{0},\Delta }| =  \frac{1}{L} \int_{A_{r_{0},\Delta  }} \bigg(\!-\frac{\gamma}{|x|} - u^1_{0}(x)\!\bigg). 
\end{align*}
Absent of surface tension ($\gamma =0$) and in light of the initial data in  \eqref{eq:jumpExInitialTemp}, this gives 
\begin{align}\label{eq:jumpSizeEquation}
  |A_{r_0,\Delta }| =  \frac{1}{L} \int_{A_{r_0,\Delta  }}  |u^1_{0}|  \;\;\Longleftrightarrow \;\; \pi\big((r_0 + \Delta)^2- r_0^2\big) =  \frac{(r_0 + \Delta \wedge \delta_0)^2- r_0^2}{L ((r_0 + \delta_0)^2- r_0^2)}. 
\end{align}
Solving numerically, we obtain $\Delta r(0)  \approx 0.22$. 
See also \cref{fig:2D2PJumpEq} which compares the area $|A_{r_0,\Delta }|$  with the excess temperature $\frac{1}{L} \int_{A_{r_0,\Delta  }}  |u^1_{0}|$ as functions of $\Delta$. We finally observe  in \cref{fig:2D2PJump} that the initial jump produced by the deep level-set method nearly matches the physical jump size. 
\begin{figure}[H]
    \centering
    \caption{Magnitude of the jump in the radius obtained by equating the excess temperature (blue curve) to the volume absorbed (red curve), corresponding to the left-hand and the right-hand sides of \eqref{eq:jumpSizeEquation}, respectively. As can be seen, the initial jump size is roughly equal to $\Delta r(0) \approx 0.22$ (precisely $0.2208$). }
    \includegraphics[height = 2in,width = 2.8in]{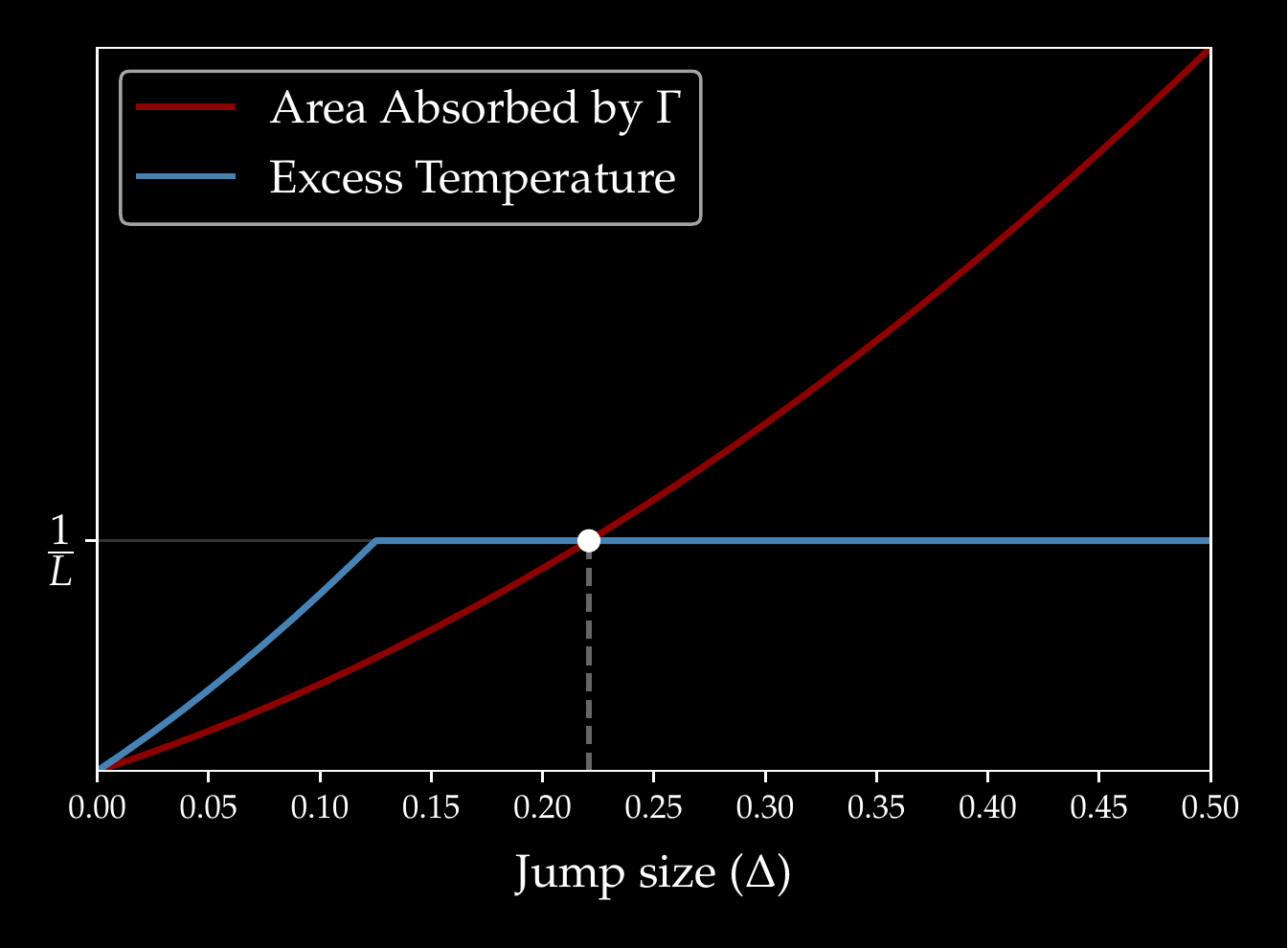}
    \label{fig:2D2PJumpEq}
\end{figure}

\begin{figure}[H]
    \centering
    \caption{Two-dimensional, two-phase supercooled Stefan problem without surface tension. The initial jump has size $\approx 0.21$ (precisely $0.2098$) and is close to the physical jump $\Delta r(0) \approx 0.22$. }
    \includegraphics[height = 2.0in,width = 2.8in]{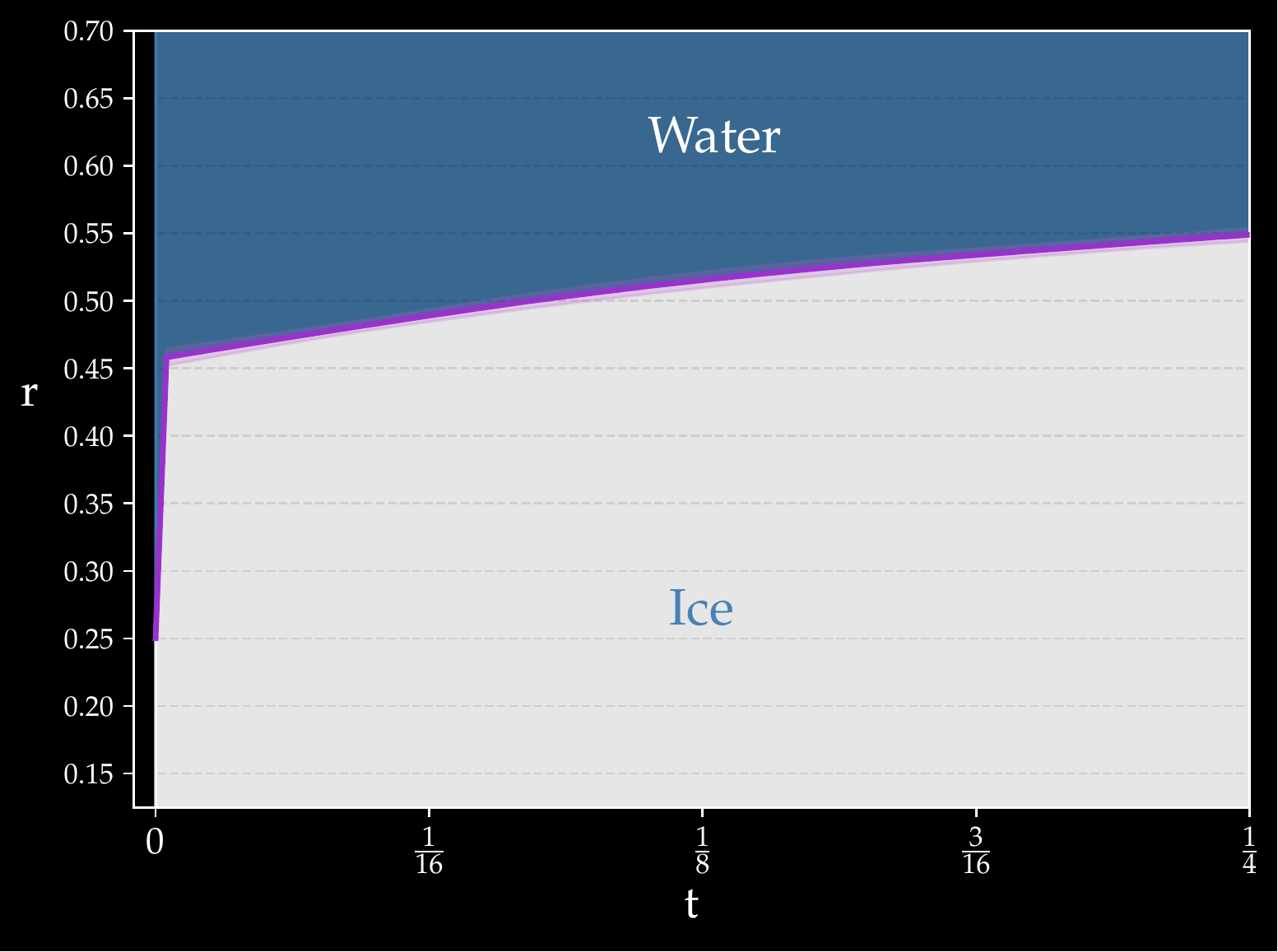}
    \label{fig:2D2PJump}
\end{figure}

\subsection{General Case}\label{sec:GeneralNumerics}

\subsubsection{Two-phase, Square-shaped Solid \textnormal{($d=2$)}}\label{sec:Square}
Let $\Omega = B_1 \subseteq \R^2$ and suppose that $\Gamma_{0-}$ is the two-dimensional  $\ell_1$-ball of radius $r_0 \in (0,1)$. For the initial level-set function, we consider $\Phi_0(x) = |x|_1 - r_0$, $|x|_1 = \sum_{i=1}^2|x_i|$. The temperature is initially uniform in both the liquid and the solid, namely  
$$u^1_0 = -c_1\frac{\mathds{1}_{\Omega \setminus \Gamma_{0-}}}{|\Omega \setminus \Gamma_{0-}|}, \quad u^2_0 = -c_2\frac{\mathds{1}_{\Gamma_{0-}}}{ |\Gamma_{0-}|}.$$

We choose 
$ \ \alpha_1 = 1/2,  \ \alpha_2 = 1/20 $, $L= 1/100$, and $\gamma = 0$.
This example intends to investigate the situation where the liquid particles have a much larger diffusivity than the solid particles. 
As the liquid is not supercooled, we expect the solid region to melt faster than with similar diffusivity.

As can be seen in  \cref{fig:2PhaseSquareNoTension}, the solid region becomes rounder as the liquid particles are more likely to hit the corners of the initial square-shaped region. Moreover, the  particles in the solid typically take longer to hit the interface as their diffusivity is far less than the one of the liquid particles.

\begin{figure}[H]
     \centering
    \caption{Two-phase Stefan problem with square-shaped initial solid region  ($d=2$). Evolution of the solid over time (no surface tension).}
    \includegraphics[height = 1.3in,width = 5.8in]{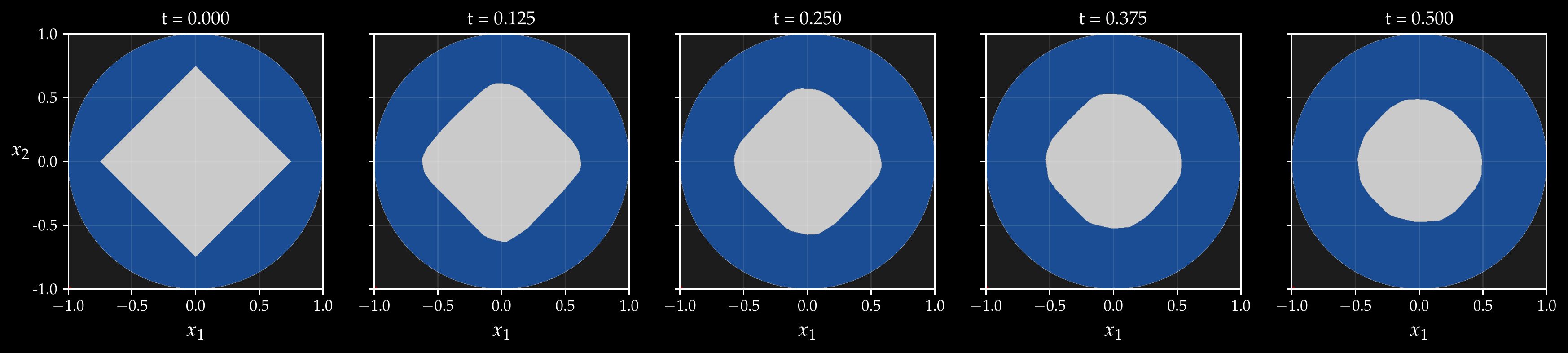}
    \label{fig:2PhaseSquareNoTension}
\end{figure}

\vspace{-5mm}

\begin{remark}
Recall that the neural network learns the difference between the level-set function $\Phi(t,\cdot)$, $t\in [0,T]$ and its initial value $\Phi_0(\cdot)$. However, $\Phi_0$ is not smooth in the following examples, so the lack of regularity  will carry over to $\Phi(t,\cdot)$ for $t\in(0,T]$. To give more flexibility to the deep level-set function, one could consider 
 $$  \Phi(t,x;\theta) := G(t,x;\theta) + \Phi_0(x)(1-\widetilde{G}(t;\widetilde{\theta}) )_+  , \quad t\in [0,T], $$
 where $\widetilde{G}\!:[0,T]\times\widetilde{\Theta} \to \R$ is another feedforward neural network taking only time as an input and capturing the decay of the initial level-set function. However, this additional flexibility did not prove helpful in our numerical experiments. One possible explanation is that the jump penalty prevents the auxiliary neural network from taking significant non-zero values.  
\end{remark}

\subsubsection{Two-phase, Diamond-shaped Solid \textnormal{($d=2$)}, Melting Regime}\label{sec:DiamondMelt}

 We are given a diamond-shaped solid $\Gamma_{0-}$  defined as the zero sublevel set of 
\begin{equation}\label{eq:levelSetDiam}
\Phi_0(x) = \bigg(\sum_{i=1}^2|x_i|^{1/2}\bigg)^2 - r_0. 
 \end{equation}
  See the leftmost panel of \cref{fig:DiamondMeltNoTension} for an illustration. We choose $r_0 = 3/4$ and the temperatures
$$u^1_0 = \frac{\mathds{1}_{\Omega \setminus \Gamma_{0-}}}{|\Omega \setminus \Gamma_{0-}|}, \qquad u^2_0 = -\frac{\mathds{1}_{\Gamma_{0-}}}{ 4|\Gamma_{0-}|}.$$
 We also set  $\eta = 1$, $L=1/4$, and $\gamma = 0.15$.   \cref{fig:DiamondMeltNoTension} displays the melting of the solid. As can be seen, the spikes of the diamond get rounder and it eventually becomes almost radially symmetric. This is because the liquid particles hit the interface more frequently near the spikes, having more contact points, which drives the melting process there.  

\begin{figure}[H]
\caption{Two-phase Stefan problem with surface tension $(\gamma = 0.15)$.  Melting of a diamond-shaped solid region  ($d=2$).}
     \centering
   \includegraphics[height = 1.3in,width = 5.8in]{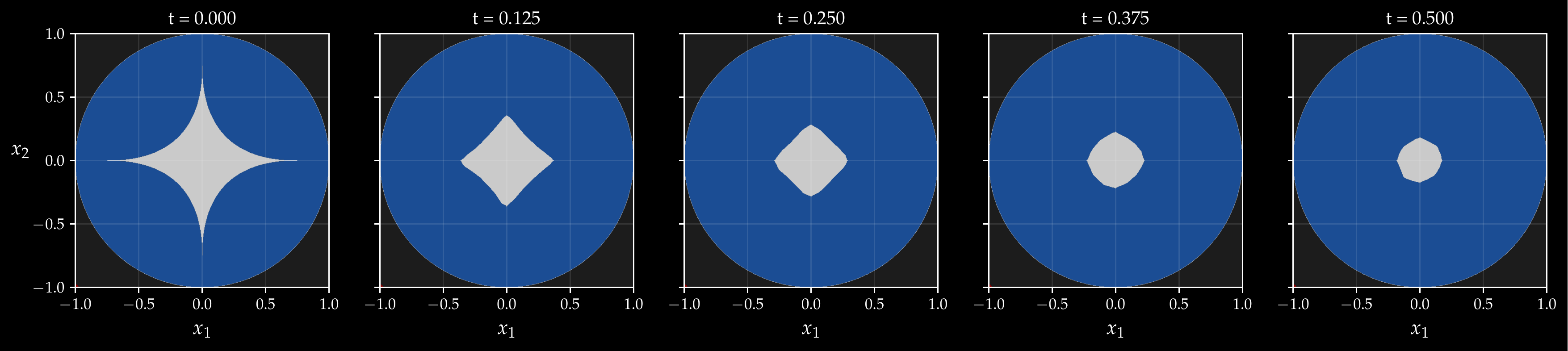}
    \label{fig:DiamondMeltNoTension}
\end{figure}

\begin{figure}[h]
\caption{Two-phase Stefan problem with diamond-shaped initial solid region  ($d=2$) and supercooled liquid. Evolution of the solid  region over time.}
\begin{subfigure}[b]{1\textwidth}
     \centering
    \caption{Without surface tension. }
    \includegraphics[height = 1.3in,width = 5.8in]{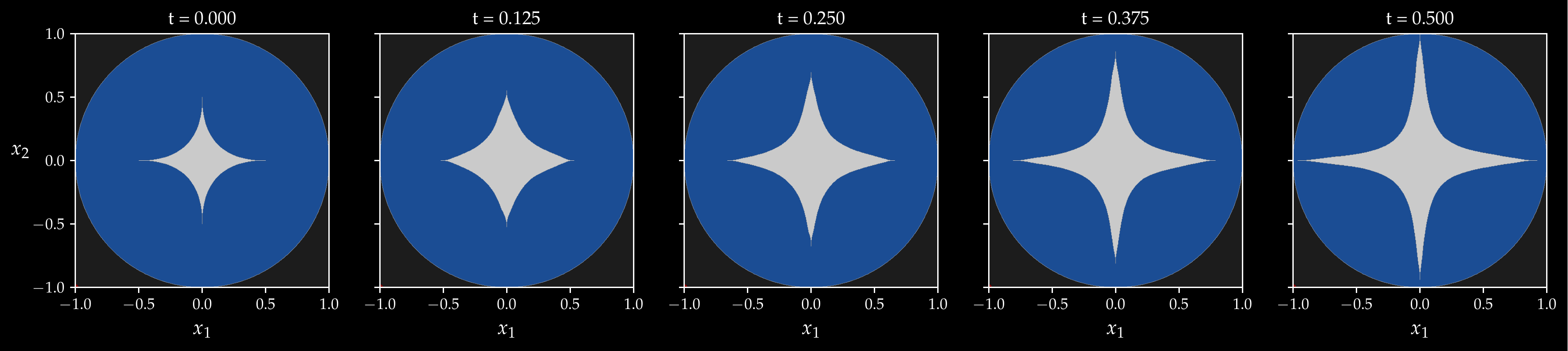}
    \label{fig:DiamondNoTension}
\end{subfigure}

\begin{subfigure}[b]{1\textwidth}
     \centering
    \caption{With surface tension $(\gamma = 0.15)$. }
    \includegraphics[height = 1.3in,width = 5.8in]{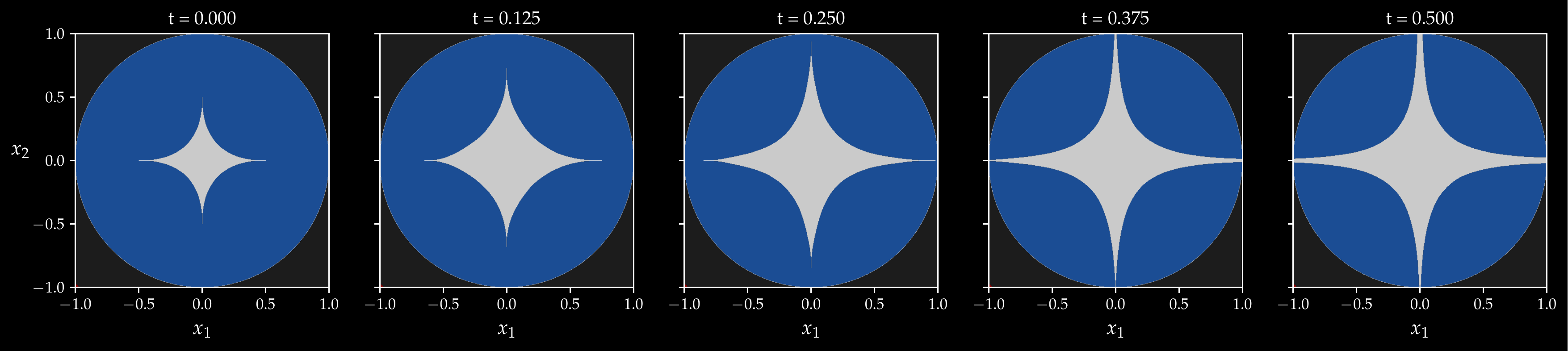}
    \label{fig:DiamondTension}
\end{subfigure}
\end{figure}
\subsubsection{Two-phase, Diamond-shaped Solid \textnormal{($d=2$)}, Freezing Regime}\label{sec:DiamondFreeze}

Consider as in \cref{sec:DiamondMelt} a diamond-shaped solid $\Gamma_{0-}$  with level-set function as in \eqref{eq:levelSetDiam}. The initial temperature is again uniform  but the liquid is supercooled this time, namely  
$$u^1_0 = -c_1\frac{\mathds{1}_{\Omega \setminus \Gamma_{0-}}}{|\Omega \setminus \Gamma_{0-}|}, \qquad u^2_0 = -c_2\frac{\mathds{1}_{\Gamma_{0-}}}{ |\Gamma_{0-}|}.$$
We choose $r_0 = 1/2$, $c_1 = 1$, $c_2=1/10$. The purpose of this example is to see the impact of surface tension on the growth of the solid. 
Recall from \cref{rem:curvconjecture} that for two-dimensional problems, only the  sign of the  curvature at the interface needs to be estimated (using \cref{alg:curv2D}). 

\cref{fig:DiamondNoTension,fig:DiamondTension} show the evolution of the solid region for the surface tension coefficients $\gamma = 0$ (no tension) and $\gamma = 0.15$, respectively. Observe that the interface is strictly concave almost everywhere (except at the spikes). In light of the Gibbs-Thomson condition \eqref{eq:GibbsThomson}, the temperature at the boundary is \textit{above} zero almost everywhere when surface tension is added, which accelerates the growth of the solid region. This is confirmed in \cref{fig:DiamondTension}. 

\subsubsection{Two-phase, Dumbbell-shaped Solid \textnormal{$(d=2)$}} \label{sec:Dumbbell}
We finally investigate the melting of a dumbbell-shaped  solid  as in the top-left panel of \cref{fig:dumbellTens2D2P}. The following parameters are used: 
$$\gamma = 0.1, \quad c_1 = 2, \quad c_2 = 0.25, \quad \eta = 1. $$ 
 Again, the initial temperature is constant in both the liquid and the solid. 
As in the previous section, we use \cref{alg:curv2D} to determine the sign of the curvature along the boundary.  \cref{fig:dumbellTens2D2P} describes the short term evolution of the solid region. As can be seen, the solid disconnects in the time interval $[0.06,0.07]$. We also note that the middle of the ``bar'' melts faster than its extremities, namely  the concave corners. 
Indeed,  concave areas  have a melting point \textit{above} zero, which slows down the melting process. 

\begin{figure}[H]
\caption{Two-phase Stefan problem with surface tension and dumbbell-shaped solid region  ($d=2$). The solid  disconnects between $t=0.06$ and $t=0.07$. }
\begin{subfigure}[b]{1\textwidth}
     \centering
    \includegraphics[height = 1.4in,width = 5.8in]{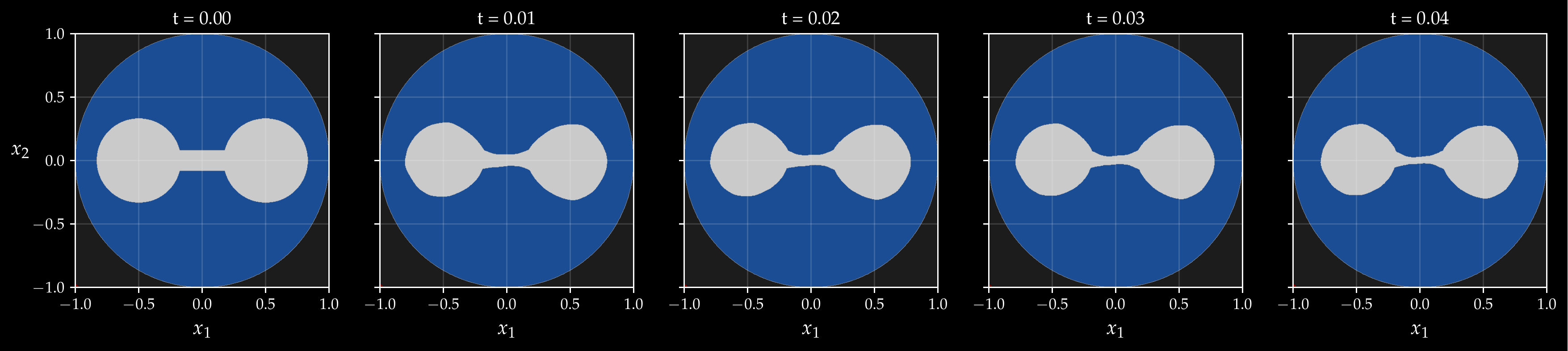}
\end{subfigure}

\begin{subfigure}[b]{1\textwidth}
     \centering
    \includegraphics[height = 1.4in,width = 5.8in]{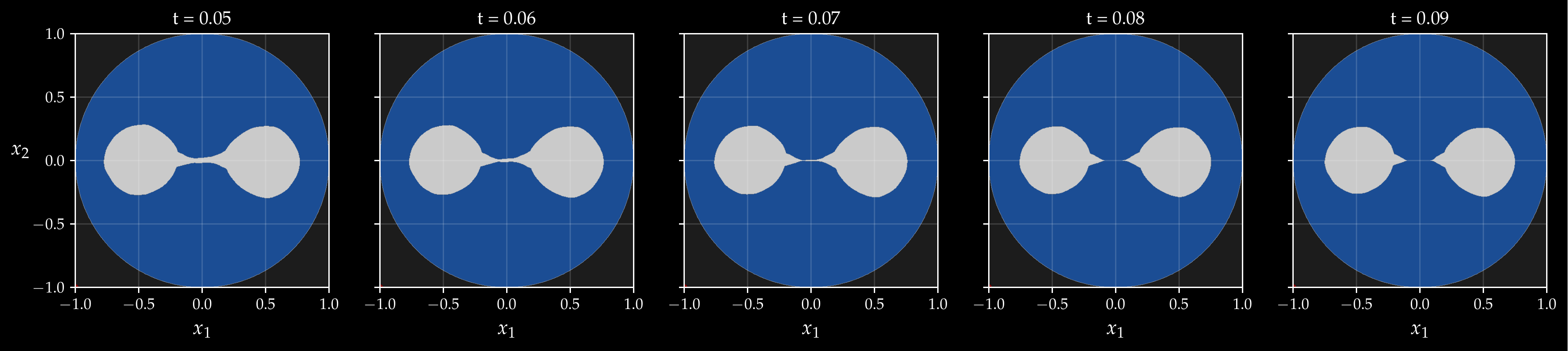}
\end{subfigure}
\label{fig:dumbellTens2D2P}
\end{figure}



\section{Conclusion} \label{sec:conclusion}
We combine the level-set method  with deep learning to represent the solid region in two-phase Stefan problems. The growth condition of probabilistic solutions is turned into a loss function and estimated using Monte Carlo simulation. The parameters of the deep level-set function are then trained using a relaxed formulation.  When adding surface tension, the algorithm involves the simulation of particles close to the solid-liquid interface and the computation of the mean curvature using a  dilation technique. The latter only requires the outward normal vector to the boundary, which is  readily available through automatic differentiation of the deep level-set function. 

The numerical examples illustrate the validity and flexibility of the method. In the two-dimensional radial case, the algorithm is capable of accurately capturing known melting rates, 
long-term radius, 
and initial jumps in the radius.
Further, the three-dimensional radial example of \cref{sec:3D2PTension} accurately demonstrates the effect of surface tension. 
We subsequently apply the method to general examples in two dimensions, where the initial solid is square-, 
diamond-,   
or dumbbell-shaped. 
The effect of surface tension is also discussed. 

We believe that our findings make a first step towards the computation and understanding of probabilistic solutions for general Stefan problems. Naturally, other examples can be explored. For instance, it would be interesting to see if the method can handle more complex dynamics with multiple jumps (e.g., Figure $1$ in \cite{MishaRadialTension})  
or capture dendritic solidification \cite{Almgren,Gibou,OsherFedkiw}. 
Another direction of future research is to investigate non-radial three-dimensional examples, such as  cubic or dumbbell-shaped solid regions. 

More broadly, the deep level-set method can be applied to other free boundary problems as well, e.g., the Hele-Shaw and the Stokes flow \cite{Howison}.  In mathematical finance, the method can be adapted to describe the stopping region of American options, thus extending the ``neural optimal stopping boundary method" \cite{ReppenSonerTissotFB} when the geometric structure of the exercise boundary is unknown. Another application would consist in computing the no-trade zone in portfolio choice problems with transaction costs \cite{MKRS}. 
\appendix
\section{Proofs}\label{app:Proofs}

\subsection{\cref{prop:GrowthCondNoTension}}
\label{app:GrowthCondNoTension}

\begin{proof} 
Let $(u^1,u^2,\Gamma)$ be  a classical solution of \eqref{eq:2SPHeat} $\!-$\eqref{eq:2SPDirichlet}. Without loss of generality, we prove the claim for $\eta =1$.   
First, the Stefan condition and the divergence theorem give 
\begin{align*}
   L \  \frac{\mathrm{d}}{\mathrm{d}t}  \int_{\Gamma_t } \psi &= L \int_{\partial \Gamma_t} V \psi  \\
   &= \frac{\alpha_2}{2} \int_{\partial \Gamma_t} \partial_{\nu}u^2\ \psi - \frac{\alpha_1}{2} \int_{\partial \Gamma_t} \partial_{\nu}u^1\ \psi \\
   &= \sum_{i=1}^2 \frac{\alpha_i}{2} \int_{\Gamma_t^i}\text{div}(\nabla u^i\  \psi) \\
   &= \sum_{i=1}^2  \frac{\alpha_i}{2} \int_{\Gamma_t^i}(\Delta u^i \ \psi +   \nabla u^i \cdot \nabla \psi).
\end{align*}
Observing that the velocity of $\Gamma^1_t$ is $-V$, we obtain
\begin{align*}
    \int_{\Gamma_t^i} \frac{\alpha_i}{2}\Delta u^i  \psi = \int_{\Gamma_t^i} \partial_t u^i\  \psi = \frac{\mathrm{d}}{\mathrm{d}t} \int_{\Gamma_t^i} u^i  \psi - (-1)^{i} \int_{\partial \Gamma_t^i} V u^i  \psi. 
\end{align*}
In light of the Dirichlet boundary condition \eqref{eq:2SPDirichlet}, 
the last (surface) integral vanishes. 
Moreover, integration by parts gives 
\begin{align*}
    \frac{\alpha_i}{2} \int_{\Gamma_t^i} \nabla u^i \cdot \nabla \psi =  (-1)^{i} \frac{\alpha_i}{2} \int_{\partial \Gamma_t^i} u^i  \partial_{\nu}\psi -  \frac{\alpha_i}{2}\int_{\Gamma_t^i} u^i   \Delta \psi = -  \frac{\alpha_i}{2}\int_{\Gamma_t^i} u^i   \Delta \psi, 
\end{align*}
invoking again the Dirichlet boundary condition. Thus, 
\begin{align}  \label{eq:intermediateGrowthCond}
   L \  \frac{\mathrm{d}}{\mathrm{d}t}  \int_{\Gamma_t } \psi =  \sum_{i=1}^2 \bigg(\frac{\mathrm{d}}{\mathrm{d}t} \int_{\Gamma_t^i} u^i  \psi -  \frac{\alpha_i}{2}\int_{\Gamma_t^i} u^i   \Delta \psi \bigg).  
\end{align}
For $t\ge 0$, define
$\overleftarrow{\; \tau}^i_{t} = \inf \{s \ge 0 \! : X_0^i+\sqrt{\alpha_i}\ W^i_s + l^i_s \notin \Gamma^{i}_{t-s} \}$.  
Note that, by a Feynman-Kac formula, 
\begin{align*}
    u^i(t,x) =  \E^{\Pb_x}[  u_0^i(X_{t}^i)\,\mathds{1}_{\{\overleftarrow{\; \tau}^i_{t} > t\}}], \quad x \in \Gamma_t^i, 
\end{align*}
where the subscript in  $\Pb_x$ indicates that  $ X_0^i = x$. 
Using the time reversibility of Brownian motion, It\^o's formula and the definition of  $\tau^i$ in \eqref{eq:rBM}, 
we obtain 
\begin{align*}
 (-1)^{i+1}\int_{\Gamma_t^i} u^i  \psi &= \int_{\Gamma_t^i} \E^{\Pb_x}[  |u_0^i(X_{t}^i)|\,\mathds{1}_{\{\overleftarrow{\; \tau}^i_{t} > t\}}]\  \psi(x)\, \mathrm{d}x \\
 &= \int_{\Gamma_t^i} \int_{\Gamma_{0}^i}   |u_0^i(y)| \Pb_x (X_{t}^i \in \mathrm{d}y, \overleftarrow{\; \tau}^i_{t} > t) \  \psi(x)\,\mathrm{d}x \\
 &= \int_{\Gamma_{0}^i} \int_{\Gamma_t^i}     \psi(x)  \Pb_y (X_{t}^i \in \mathrm{d}x, \tau^i > t) \ |u_0^i(y)|\,\mathrm{d}y \  \\[0.8em]
  &= \E^{\mu^i}[  \psi(X_{t}^i)\,\mathds{1}_{\{ \tau^i > t\}}] \hspace{5cm} 
  \\[0.8em] 
   &=
   \E^{\mu^i}[  \psi(X_{\tau^i \wedge t}^i)]  -   \E^{\mu^i}[  \psi(X_{\tau^i}^i)\,\mathds{1}_{\{ \tau^i \le t\}}] \\[1em]
   &= \E^{\mu^i}[\psi(X^i_0)] +  \frac{\alpha_i}{2} \int_0^{ t} \underbrace{\E^{\mu^i}[    \Delta  \psi(X_{s}^i)\, \mathds{1}_{\{\tau^i > s\}}] }_{ =  (-1)^{i+1} \int_{\Gamma_s^i} u^i  \Delta \psi}  \,\mathrm{d}s -   \E^{\mu^i}[  \psi(X_{\tau^i}^i)\,\mathds{1}_{\{ \tau^i \le t\}}].
\end{align*}
Hence $\frac{\mathrm{d}}{\mathrm{d}t} \int_{\Gamma_t^i} u^i  \psi -  \frac{\alpha_i}{2}\int_{\Gamma_t^i} u^i   \Delta \psi = (-1)^{i}\frac{\mathrm{d}}{\mathrm{d}t}\E^{\mu^i}[  \psi(X_{\tau^i}^i)\,\mathds{1}_{\{ \tau^i \le t\}}]$ and the result follows.

\end{proof}


\subsection{\cref{prop:GrowthCondTension}}
\label{app:GrowthCondTension}

\begin{proof}
It is enough to prove the claim for $\eta=1$. Looking at the proof of \cref{prop:GrowthCondNoTension} and using that $u^1,u^2$ coincide on $\partial \Gamma$, we gather  that
\begin{align*}
    L \ \frac{\mathrm{d}}{\mathrm{d}t} \int_{\Gamma_t } \psi
    &= \sum_{i=1}^2 \bigg( \frac{\mathrm{d}}{\mathrm{d}t} \int_{\Gamma_t^i} u^i  \psi - (-1)^i \int_{\partial \Gamma_t} V u^i  \psi +  \frac{\alpha_i}{2} \int_{\Gamma_t^i} \nabla u^i \cdot \nabla \psi \bigg) \\ 
   &= \sum_{i=1}^2 \bigg( \frac{\mathrm{d}}{\mathrm{d}t} \int_{\Gamma_t^i} u^i  \psi -  \frac{\alpha_i}{2}\int_{\Gamma_t^i} u^i   \Delta \psi \bigg) -   \dot{\calK}_t^{I},
\end{align*}
where $ \dot{\calK}_{t}^{I} \ =  \gamma\frac{\alpha_2-\alpha_1}{2} \int_{\partial \Gamma_t} \kappa_{\partial \Gamma_t} \  \partial_{\nu} \psi$.
Moreover, $u^i$ admits the Feynman-Kac representation 
\begin{align*}
    u^i(t,x) =   \E^{\Pb_x}[  u_0^i(X_{t}^i)\,\mathds{1}_{\{\overleftarrow{\; \tau}^i_{t} > t\}}]  - \gamma  K^i(t,x), \quad x \in \Gamma_t^i, 
\end{align*}
with $K^i$ and $\overleftarrow{\; \tau}^i_{t}$  given in the statement. This implies that
\begin{align*}
   L \  \frac{\mathrm{d}}{\mathrm{d}t} \int_{\Gamma_t } \psi  = \sum_{i=1}^2 
(-1)^{i}\frac{\mathrm{d}}{\mathrm{d}t}\E^{\mu^i}[  \psi(X_{\tau^i}^i)\,\mathds{1}_{\{ \tau^i \le t\}}]  \ - \dot{\calK}_{t}^{I} -  \dot{\calK}_{t}^{II} +  \dot{\calK}_{t}^{III},
\end{align*}
with  the additional terms  
\begin{align}\label{eq:extraterm}
     \dot{\calK}_{t}^{II} \ =  \gamma\sum_{i=1}^2 \frac{\mathrm{d}}{\mathrm{d}t} \int_{\Gamma_t^i}  K^i \psi, \quad 
     \dot{\calK}_{t}^{III} \ =   \gamma\sum_{i=1}^2\frac{\alpha_i}{2}\int_{\Gamma_t^i} K^i \Delta \psi.
\end{align}
Next, due to \eqref{eq:curvK}, $K^i$ is  a weak solution of the heat equation in any rectangle $\calR \subseteq \text{int} \ \Gamma^i$, i.e., 
$$\int_{\calR}\Big(\partial_t + \frac{\alpha_i}{2}\Delta\Big)\psi \ K^i = 0, \quad \psi \in \calC_c^{\infty}(\calR). $$
From Weyl's Lemma, we conclude that $K^i$ is in fact a strong solution of  $\partial_t K^i = \frac{\alpha_i}{2} \Delta K^i$ in $\text{int} \ \Gamma^i$. 
Noting also that $K^1 = K^2$ on $\partial \Gamma$, we have
 \begin{align*}
       \frac{1}{\gamma}\dot{\calK}_t^{II} \ = \sum_{i=1}^2 \frac{\mathrm{d}}{\mathrm{d}t} \int_{\Gamma_t^i}  K^i \psi
     = \sum_{i=1}^2 \bigg( \int_{\Gamma_t^i}  \partial_t K^i \ \psi + (-1)^{i}\int_{\partial \Gamma_t}  V K^i \psi  \bigg) =  \sum_{i=1}^2  \int_{\Gamma_t^i}  \frac{\alpha_i}{2}\Delta K^i \ \psi.
 \end{align*}
 Using Green's second identity, note that (again, the outward normal of $\Gamma_t^1$ is  $-\nu$)
  \begin{align*}
    \int_{\Gamma_t^i}  \Delta K^i \ \psi =   (-1)^{i}\int_{\partial \Gamma_t}  (\partial_{\nu} K^i \ \psi -  K^i \ \partial_{\nu}\psi)  + \int_{ \Gamma_t^i}   K^i  \Delta \psi. 
 \end{align*}
Given the definition of $\dot{\calK}_t^{III}$ in \eqref{eq:extraterm} and recalling that  $K^1|_{\partial\Gamma} = K^2|_{\partial\Gamma} = \kappa_{\partial \Gamma}$, this  implies that 
   \begin{align*}
         \dot{\calK}_t^{II} -\dot{\calK}_t^{III} &= \gamma \sum_{i=1}^2  (-1)^{i}\frac{\alpha_i}{2}\int_{\partial \Gamma_t}  (\partial_{\nu} K^i \ \psi -  \kappa_{\partial \Gamma_t} \ \partial_{\nu}\psi) \\ &= -\dot{\calK}_t^{I} + \gamma\int_{\partial \Gamma_t} \bigg(\frac{\alpha_2}{2}\partial_{\nu} K^2 -  \frac{\alpha_1}{2}\partial_{\nu} K^1 \bigg) \psi\\
      &= -\dot{\calK}_t^{I} + \frac{\mathrm{d}}{\mathrm{d}t}\calK_t.
 \end{align*}
Integrating in time and rearranging yields the claim. 
\end{proof}


\subsection{\cref{prop:growthCondTensionPoisson}}
\label{app:growthCondTensionPoisson}

\begin{proof}  Assume again that $\eta=1$. Using similar arguments to the ones in the proof of \cref{prop:GrowthCondNoTension}, we can show that 
\begin{align}\label{eq:growthTensionproof1}
    L\bigg(\int_{\Gamma_{0-}}\psi -  \int_{\Gamma_t  } \psi \bigg) &=  \sum_{i=1}^2 \bigg( \int_{\Gamma_{0-}^i} u^i \psi - \int_{\Gamma_t^i} u^i  \psi +  \frac{\alpha_i}{2}\int_0^t\int_{\Gamma_s^i} u^i   \Delta \psi \bigg).
\end{align}
Expressing $u := -u^1\,\mathds{1}_{\Gamma_t^1} - u^2\,\mathds{1}_{\Gamma_t^2}$ via \cite[Lemma 2.4]{MishaRadialTension}, we find for $\varphi \in \{\psi,\Delta \psi\}$,  
\begin{align}
  (-1)^{i+1} \int_{\Gamma_t^i} u^i \varphi &= \E[\varphi(X_t^i)\, \mathds{1}_{\{\tau^{i}>t\}}] + (-1)^i\ \lim_{\delta \downarrow 0} \  \E\bigg[\sum_{l=1}^{  \infty} \varphi(Y_{t}^{l,i})\,\mathds{1}_{\{T_l^{\delta,i}\le t< \tau^{\delta,i}_l\}} \bigg] \label{eq:intu^i}.
\end{align}
Next,  let us temporarily write  $(Y,T_l,\tau_l,\alpha) = (Y^{l,i}, T_l^{\delta,i}, \tau_{l}^{\delta,i},\alpha_i)$ to simplify notation. Using Itô's lemma and Fubini's theorem, observe that 
\begin{align*}
   \psi(Y_{t})\,\mathds{1}_{\{ T_l\le t< \tau_l\}}  &=  \psi(Y_{\tau_l \wedge t})\,\mathds{1}_{\{ T_l\le t\}}  
   -  \psi(Y_{\tau_l})\,\mathds{1}_{\{\tau_l \le t\}} \\[1em] 
    &= \psi(Y_{T_l})\,\mathds{1}_{\{ T_l\le t\}} + \frac{\alpha}{2}\int_{0}^{t} \Delta \psi(Y_{s})\,\mathds{1}_{\{T_l \le s<\tau_l \}}\,\mathrm{d}s \\
    &\quad + \int_{T_l\wedge t}^{\tau_l\wedge t} \nabla \psi(Y_{s}) \cdot \mathrm{d}Y_s - \psi(Y_{\tau_l})\,\mathds{1}_{\{\tau_l \le t\}} \\[1em] 
    \Longrightarrow \;\; \E[\psi(Y_{t})\,\mathds{1}_{\{ T_l\le t< \tau_l\}}] &= \E[ \psi(Y_{T_l})\,\mathds{1}_{\{ T_l\le t\}} - \psi(Y_{\tau_l})\,\mathds{1}_{\{\tau_l \le t\}}] + \frac{\alpha}{2}\int_{0}^{t} \E[\Delta \psi(Y_{s})\,\mathds{1}_{\{ T_l \le s < \tau_l \}}]\,\mathrm{d}s. 
\end{align*}
Similarly, for $X_t^i$ and $\tau^i$, we have 
\begin{align*}
   \E[\psi(X_{t}^i)\,\mathds{1}_{\{\tau^i > t\}}]  &= \E[\psi(X_{0}^i)-\psi(X_{\tau^i}^i)\,\mathds{1}_{\{\tau^i \le t\}}]\ + \frac{\alpha_i}{2}\int_{0}^{t} \E[\Delta \psi(X_{s}^i)\, \mathds{1}_{\{\tau^i > s  \}}]\,\mathrm{d}s .
\end{align*}
 Plugging the above expressions into \eqref{eq:intu^i}, with $\varphi = \psi$, and invoking Fubini's theorem, we obtain 
\begin{align*}
    (-1)^{i+1} \bigg(\int_{\Gamma_{0-}^i} u^i \psi - \int_{\Gamma_t^i} u^i  \psi \bigg) &=  -\frac{\alpha_i}{2}\int_{0}^{t} \E[ \Delta \psi(X_{s}^i)\, \mathds{1}_{\{\tau^i > s  \}}]\,\mathrm{d}s +  \E[\psi(X_{\tau^i}^i) \, \mathds{1}_{\{\tau^i \le t\}}]  \\[1em]
     &\quad - (-1)^i\ \lim_{\delta \downarrow 0} \   \E\bigg[\sum_{l = 1}^{\infty } \Big( \psi(Y_{T_l^{\delta,i}}^{l,i})\,\mathds{1}_{\{T_l^{\delta,i}\le t\}}- \psi(Y_{\tau_l^{\delta,i}}^{l,i}) \,\mathds{1}_{\{\tau^{\delta,i}_l \le t\}} \Big) \bigg] \\[1em]
    &\quad - (-1)^i\ \lim_{\delta \downarrow 0} \  \frac{\alpha_i}{2} \int_0^t\E\bigg[\sum_{l = 1}^{\infty } \Delta \psi(Y_{s}^{l,i})\,\mathds{1}_{\{T_l^{\delta,i}\le s < \tau^{\delta,i}_l\}} \bigg]\,\mathrm{d}s\\[1em]
    &= (-1)^i\ \frac{\alpha_i}{2}\int_0^t\int_{\Gamma_s^i} u^i   \Delta \psi \ \,\mathrm{d}s + \E[\psi(X_{\tau^i}^i)\, \mathds{1}_{\{\tau^i \le t\}}] + (-1)^i\calK_t^i, 
\end{align*}
using  \eqref{eq:intu^i}  with $\varphi = \Delta \psi$  for the last equality. In view of   \eqref{eq:growthTensionproof1} and the definition of $\calK$, $(N^{\delta,i})_{i=1}^2$ in the statement, we indeed find that 
\begin{align*}
  \int_{\Gamma_{0-}}\psi -  \int_{\Gamma_t  } \psi  &=  \frac{1}{L}\Big( \E^{\mu^1}[ \psi(X_{\tau^1}^1)\,\mathds{1}_{\{ \tau^1 \le t\}}] - \E^{\mu^2}[  \psi(X_{\tau^2}^2)\,\mathds{1}_{\{ \tau^2 \le t\}}] + \calK_{t}   \Big).
\end{align*}
\end{proof}

\bibliographystyle{abbrvnat}
\bibliography{ref.bib}

\end{document}